\documentclass[12pt]{amsart}

\usepackage[margin=1.20in]{geometry}

\usepackage{amssymb,amsmath}
\usepackage{amsthm,enumitem}
\usepackage{hyperref}
\usepackage{mathrsfs}
\usepackage{textcomp}

\newtheorem{thm}{Theorem}[section]
\newtheorem{lem}[thm]{Lemma}
\newtheorem{prop}[thm]{Proposition}

\newtheorem{cor}[thm]{Corollary}

\numberwithin{equation}{section}

\def\Xint#1{\mathchoice
  {\XXint\displaystyle\textstyle{#1}}%
  {\XXint\textstyle\scriptstyle{#1}}%
  {\XXint\scriptstyle\scriptscriptstyle{#1}}%
  {\XXint\scriptscriptstyle\scriptscriptstyle{#1}}%
  \!\int}
\def\XXint#1#2#3{{\setbox0=\hbox{$#1{#2#3}{\int}$}
  \vcenter{\hbox{$#2#3$}}\kern-.5\wd0}}

\def\dashint{\Xint-}

\newcommand{\al}{\alpha}                \newcommand{\lda}{\lambda}
\newcommand{\om}{\Omega}                \newcommand{\pa}{\partial}
\newcommand{\va}{\varepsilon}           \newcommand{\ud}{\mathrm{d}}
\newcommand{\be}{\begin{equation}}      \newcommand{\ee}{\end{equation}}
\newcommand{\w}{\omega}                 
\newcommand{\Lda}{\Lambda}              \newcommand{\B}{\mathcal{B}}
\newcommand{\R}{\mathbb{R}}              \newcommand{\Sn}{\mathbb{S}^n}

\newcommand{\abs}[1]{\lvert#1\rvert}

\begin{document}

\title[A dual Yamabe flow]{A dual Yamabe flow and related integral flows}

\author[J. Xiong]{Jingang Xiong}
\address{School of Mathematical Sciences, Laboratory of Mathematics and Complex Systems, MOE\\ Beijing Normal University, Beijing 100875, China}
\email{jx@bnu.edu.cn}
\thanks{Partially supported by NSFC 12325104 and NSFC 12271028.}

%\subjclass[2020]{Primary ; Secondary }


\date{\today}

\begin{abstract}
We study a family of nonlinear integral flows that involve Riesz potentials on Riemannian manifolds. In the Hardy-Littlewood-Sobolev (HLS) subcritical regime, we present a precise blow-up profile exhibited by the flows. In the HLS critical regime, by introducing a \textit{dual $Q$ curvature} we demonstrate the concentration-compactness phenomenon. If, in addition, the integral kernel matches with the Green's function of a conformally invariant elliptic operator, this critical flow can be considered as a dual Yamabe flow. Convergence is then established on the unit spheres, which is also valid on certain locally conformally flat manifolds.
\end{abstract}

\dedicatory{In memory of my father}

\maketitle

\section{Introduction}

\subsection{Background and motivations}  Let $(M,g_0)$ be a smooth compact Riemannian manifold of dimension $n\ge 3$ and
\[
[g_0]=\{\rho g_0: \rho\in C^\infty(M),~ \rho>0 \}
\]
be the conformal class of $g_0$.  The Yamabe problem asks whether there exists a metric $g\in [g_0]$ of constant scalar curvature.
Recall that  the conformal Laplacian of $g_0$ on $M$ is given by
\[
L_{g_0}:=\Delta_{g_0} -c(n) R_{g_0},
\]
where $\Delta_{g_0}$ denotes the Laplace-Beltrami operator associated with the metric $g_0$,  $c(n)= \frac{n-2}{4(n-1)} $ and $R_{g_0}$ represents the scalar curvature of $g_0$.  It satisfies the conformal transformation law
\[
L_{u^{\frac{4}{n-2}} g_0}(\phi)= u^{-\frac{n+2}{n-2}} L_{g_0} (u\phi) \quad \forall~ \phi\in C^2(M).
\]
Hence, the Yamabe problem amounts to seeking a solution of
\[
-L_{g_0}u=c u^{\frac{n+2}{n-2}} \quad \mbox{on }M, \quad u>0,
\]
for some constant $c$.  The existence was proved by the variational method through Yamabe \cite{Y}, Trudinger \cite{T}, Aubin \cite{A} and Schoen \cite{S}.   In 1980s, R. Hamilton introduced the Yamabe flow
\[
\pa_t g= -(R_g-r_g) g
\]
where $t\ge 0$ and  $r_g= \mathrm{Vol}_g(M)^{-1} \int_M R_g\,\ud vol_g$ and $ \mathrm{Vol}_g(M)$ denotes the volume of $M$ with respect to the metric $g$. The Yamabe flow  is the normalized negative $L^2$ gradient flow of the Yamabe  functional
\[
F_2[g]:=
\frac{\int_{M} R_g\,\ud vol_g}{ (\int_{M} \,\ud vol_g )^{\frac{n-2}{n}}}
\]
and the conformal class is preserved along the flow.  Thus it has a scalar form
\[
\pa_t u^{\frac{n+2}{n-2}}= L_{g_0} u+c(n) r_g u^{\frac{n+2}{n-2}}.
\]
The convergence had been established by Chow \cite{Chow},  Ye \cite{Ye}, Schwetlick-Struwe \cite{Schwetlick&Struwe} and Brendle \cite{Brd}.

Analogously, there has been much interest in the fourth-order Yamabe-type problem.  If $n\ge 5$,  let
\begin{align*}
P_{g_0}&:=  \Delta_{g_0}^2 -\mathrm{div}_{g_0}(a_n R_{g_0} g+b_nRic_{g_0})d+\frac{n-4}{2}Q_{g_0}, \\
Q_{g_0}&:=-\frac{1}{2(n-1)} \Delta_{g_0} R_{g_0}+\frac{n^3-4n^2+16n-16}{8(n-1)^2(n-2)^2} R_{g_0}^2-\frac{2}{(n-2)^2} |Ric_{g_0}|^2,
\end{align*} be the Paneitz operator \cite{Pan83} and the $Q$ curvature \cite{Bra85}, respectively. The Paneitz operator
satisfies the conformal transformation law
\[
P_{u^{\frac{4}{n-4}}g_0}(\phi)= u^{-\frac{n+4}{n-4}} P_{g_0} (u\phi) \quad \forall~ \phi\in C^4(M).
\]
Hence, the Yamabe type problem for $Q$ curvature is equivalent to solving
\[
P_{g_0} ( u)=c u^{\frac{n+4}{n-4}} \quad \mbox{on }M, \quad u>0,
\]
for some constant $c$. However, both the variational method and potential fourth-order flow approach encounter difficulties in obtaining a positive object.

Recently, Gursky-Malchiodi \cite{GM} established the existence of solutions to the 4th order Yamabe problem under the assumption that there is a conformal metric with nonnegative scalar and $Q$ curvature, and that the $Q$ curvature is strictly positive at some points. Hang and Yang \cite{HY} demonstrated existence on manifolds of positive Yamabe type, provided there exists a conformal metric whose $Q$ curvature is nonnegative and positive at certain points.  Both assumptions of \cite{GM} and \cite{HY} imply that
\be \label{eq:positivity}
P_{g_0} \mbox{ is invertible, positive and its Green's function is positive}.
\ee Gursky-Malchiodi's proof  employs  the normalized $W^{2,2}$ gradient flow of the total $Q$ curvature functional
\[
F_4[g]= \mathrm{Vol}_g(M)^{-\frac{n-4}{n}} \int_{M} Q_g\,\ud vol_g,
\]
whereas Hang-Yang's proof leverages a dual variation problem; see Malchiodi \cite{M} for a survey. We also refer to Hebey-Robert \cite{HR}, Qing-Raske \cite{QR} for earlier results.

The flow in \cite{GM} is given
\be \label{eq:gm-flow}
\pa_t u=-u+ \mu(t)(P_{\bar g})^{-1}(|u|^{\frac{n+4}{n-2}})
\ee
and satisfies initial condition $u(\cdot, 0)=1$,
where $\bar g\in [g_0]$, $(P_{\bar g})^{-1}$ is the inverse of $P_{\bar g}$ and $\mu(t)$ is a normalization. They have shown that the flow converges to a solution of the fourth-order Yamabe problem, given that $F_4[\bar g]$ falls below a natural threshold. However, exploring the dynamic behavior when $\bar g$ is any element within the conformal class $[g_0]$ remains an intriguing problem.

In addition to the conformal Laplacian and Paneitz operator, there are many other important conformally invariant operators of fractional and higher orders. See  Fefferman-Graham \cite{FG},  Graham-Jenne-Mason-Sparling \cite{GJMS}, Graham-Zworski \cite{GZ}, Chang-Gonz\'alez \cite{Chang-G} and Chang-Yang \cite{CY17} and many others. The fractional Yamabe flow has been studied by  Jin-Xiong \cite{JX14}, Daskalopoulos-Sire-V\'azquez \cite{DSV} and Chan-Sire-Sun \cite{CSS}.

Finally, we note that differential integral flows of \eqref{eq:gm-flow} type have been frequently used to deform level sets of functionals in the critical point theory. In particular, it plays a crucial role in  Bahri-Coron \cite{BC} about the Nirenberg problem on the three dimensional unit sphere.  In this paper, we would like to conduct a detailed study of this type of flow.

\subsection{A general framework and main theorems}

Let $M$ be a compact smooth Riemannian manifold of dimension $n\ge 1$, and $K_0:M\times M\to (0,\infty]$ be a $C^1$ singular kernel of the Riesz potential type.  Namely, for any $X,Y\in M$,
\begin{align}
K_0(X,Y)&=K_0(Y,X), \tag{K-1}\label{eq:K-1}\\
\frac{1}{\Lda} d_{g_0} (X,Y)^{2\sigma-n}&\le K_0(Y,X)\le \Lda d_{g_0} (X,Y)^{2\sigma-n},\tag{K-2}\label{eq:K-2}\\
|\nabla_{g_0} K_0(\cdot,Y)|&\le \Lda dist_g (\cdot,Y)^{2\sigma-n-1} \quad \mbox{on }M\setminus \{Y\}, \tag{K-3}  \label{eq:K-3}
\end{align}
where $d_{g_0}$ is the distance function with respect to metric $g_0$,  $0<\sigma<n/2$ and $\Lda\ge 1$ are constants.

Define
\[
\mathcal{K}_{g_0}(f)(X):=\int_{M} K_0(X,Y) f(Y)\,\ud vol_{g_0}(Y) \quad \mbox{for }f\in L^1(M).
\]
For $T>0$, we study the Cauchy problem for the differential-integral equation:
\be\label{eq:IP}
\begin{cases}
\pa_t u^m &= \mathcal{K}_{g_0}(u)\quad \mbox{on }M\times (0,T),\\
u(\cdot,0)&=u_0\ge 0,
\end{cases}
\ee
where $m>0$ and $u_0\in C^{0}(M)$ is not identical to zero. By Riesz potential estimates and Picard-Lindel\"of theorem,   \eqref{eq:IP} has a unique solution with $u^m \in C^1([0,T^*); C^0(M))$ for some $T^*>0$ representing the maximum existence time.

If $m>\frac{n-2\sigma}{n+2\sigma}$, $m\neq 1$, the steady problem
\be \label{eq:st0}
\mathcal{K}_{g_0}(S)=S^m \quad \mbox{on }M, \quad  S>0 ,
\ee
has a continuous solution. Moreover, the solution is unique if $m>1$. See Section \ref{s:SS} below.

\begin{thm}\label{thm:main0} Let $u$ be a solution of \eqref{eq:IP} with $u^m \in C^1([0,T^*); C^0(M))$ for some $T^*>0$ representing the maximum existence time. Then we have:
\begin{itemize}
\item[(i)] If $m> 1$, then $T^*=\infty$ and
\[
\lim_{t\to \infty} \|t^{-\frac{1}{m-1}} u(\cdot, t)- c S(\cdot)\|_{C^0(M)}=0;
\]
\item[(ii)] If $\frac{n-2\sigma}{n+2\sigma}<m<1$, then $T^*<\infty$ and
\[
\lim_{t\to ( T^*)^-} \|(T^*- t)^{-\frac{1}{m-1}} u(\cdot, t)- c S(\cdot)\|_{C^0(M)}=0.
\]
\item[(iii)] If $m=\frac{n-2\sigma}{n+2\sigma}$, then $T^*<\infty$ and
\[
\lim_{t\to ( T^*)^-} (T^*- t)^{-\frac{1}{m-1}} \| u(\cdot, t)\|_{L^{\frac{2n}{n+2\sigma}}}= c.
\]
Here $S$ is a solution of \eqref{eq:st0} and $c$ are positive constants.
\end{itemize}

\end{thm}

One can obtain precise convergence rates for the above items (i) and (ii),  as exemplified by Jin-Xiong-Yang \cite{JXY} in their study of a nonlinear boundary control problem. Section \ref{s:conver} provides essential tools such as the linearized operator and eigenfunctions for further exploration. Our proof is inspired by previous studies on the Cauchy-Dirichlet problem for the porous medium equations; see \cite{JXY} and references therein.

For the linear case $m=1$, we also have $T^*=\infty$ and
\[
\lim_{t\to \infty} \|e^{-t} u(\cdot, t)-  \phi_1(\cdot)\|_{C^0(M)}=0,
\]
where $\phi_1>0$ is a positive  eigenfunction  associated to the largest eigenvalue of $\mathcal{K}_{g_0}$. If $0<m< \frac{n-2\sigma}{n+2\sigma}$, we still have $T^*<\infty$. But the blow-up behavior is unclear.

When $m= \frac{n-2\sigma}{n+2\sigma}$, the flow has an independent interest in conformal geometry as follows.  Let
 \[
 [g_0]_0=\{u^{\frac{4}{n+2\sigma}} g_0: u\in C^{0}(M),~u>0\}
 \]
be the $C^0$ conformal class of $g_0$. For any $g=u^{\frac{4}{n+2\sigma}} g_0\in [g_0]_0$, we let
\[
\mathcal{K}_{g}(f)(X):=\int_{M} K_g(X,Y) f(Y)\,\ud vol_{g}(Y) \quad \mbox{for }f\in L^1(M),
\]
where
\[
K_g(X,Y)= \Big(u(X)u(Y)\Big)^{-\frac{n-2\sigma}{n+2\sigma}}K_0(X,Y).
\]
By definition, $\mathcal{K}_{g}$ is a conformally invariant operator and
\[
\mathcal{K}_{u^{\frac{4}{n+2\sigma}} g_0}(\phi)= u^{-\frac{n-2\sigma}{n+2\sigma}}\mathcal{K}_{g_0}(u \phi) \quad \mbox{for all }\phi\in L^1(M).
\]
This motivates us to define
\be \label{eq:q-k}
Q_{K_0}^g: =\mathcal{K}_{g}(1)
\ee
and
\be \label{eq:total-q-k}
F_{K_0}[g]:= \mathrm{Vol}_g (M)^{-\frac{n+2\sigma}{n}} \int_{M} Q_{K_0}^g \,\ud vol_g.
\ee
We may call $Q_{K_0}^g$ a \textit{dual $Q$ curvature}, which is related to a quantity introduced by Zhu \cite{Z1} and Han-Zhu \cite{HZ}  for a specific kernel, although they are distinct.
The normalized $L^2$ gradient flow of $F_{K_0}[\cdot]$ takes the form
\be \label{eq:metric-version}
\pa_t g= (Q_{K_0}^g-  a(t)  )g,
\ee
where $t\ge 0$ and $a(t)=\int_{M}Q_{K_0}^g \,\ud vol_g$. By setting $g=u^{\frac{4}{n+2\sigma}}g_0$, we obtain
\be \label{eq:main}
\pa_t u^{\frac{n-2\sigma}{n+2\sigma }}=\mathcal{K}_{g_0}(u) - a(t)  u^{\frac{n-2\sigma}{n+2\sigma }}.
\ee
Suppose that
\be \label{eq:initial-data}
u(\cdot, 0)= u_0(\cdot) \ge 0 \quad \mbox{on }M,
\ee
and $u_0\in C^0(M) $ is not identical to zero.

If $K_0$ matches with the Green's function of some invertible conformally invariant differential operator, we may call \eqref{eq:metric-version} (or \eqref{eq:main}) a \textit{dual Yamabe flow} as it stems from the dual type functional $F_{K_0}[g]$.

On the standard $n$-sphere $\Sn$ equipped with the induced metric $g_0$ from $\R^{n+1}$, there defines the intertwining operator $ P_\sigma^{g_0}$ which is the pull-back operator of  the $\sigma\in(0,n/2)$ power of the Laplacian $(-\Delta)^{\sigma}$ on $\R^n$ via the stereographic projection:
\[
P_\sigma^{g_0}(\phi)\circ F=  |J_F|^{-\frac{n+2\sigma}{2n}}(-\Delta)^\sigma(|J_F|^{\frac{n-2\sigma}{2n}}(\phi\circ F)),
\]
where $F$ is the inverse of the stereographic projection and $|J_F|$ is the determinant of the Jacobian of $F$. Moreover,
\be\label{eq:p-inverse}
(P_{\sigma}^{g_0})^{-1}(f)(\xi)=c_{n,\sigma}\int_{\Sn}\frac{f(\zeta)}{\abs{\xi-\zeta}^{n-2\sigma}}\,\ud vol_{g_{0}}(\zeta)\quad \mbox{for }f\in L^1(\Sn),
\ee
where $c_{n,\sigma}=\frac{\Gamma(\frac{n-2\sigma}{2})}{2^{2\sigma}\pi^{n/2}\Gamma (\sigma)}$, $p\ge 1$ and $|\cdot|$ is the Euclidean distance in $\R^{n+1}$.
The classical result  of Lieb \cite{Lieb83}  asserts that
\be\label{eq:lieb}
\sup \left\{\int_{\Sn} f (P_{\sigma}^{g_0})^{-1} (f) \,\ud vol_{g_0}  : \int_{\Sn} |f|^{\frac{2n}{n+2\sigma}} \,\ud vol_{g_0} =1\right\}:=S_{n,-\sigma}
\ee
is achieved and the maximizers have to be form
\be\label{eq:bubble}
\bar U_{\xi_0,\lda }( \xi ) = \left(\frac{2\lda }{2+ (\lda^2-1) (1 - \cos d_{g_0}(\xi, \xi_0) )   } \right)^{\frac{n+2\sigma}{2}} \quad \mbox{for some }\lda>0, ~\xi_0\in \Sn
\ee
upon a sign.
Finally, we remark that
by taking $\mathcal{K}_{g_0}= (P_{\sigma}^{g_0})^{-1}$ and using the stereographic projection, \eqref{eq:main} can be transformed into the prototype
\[
\pa_tv^{\frac{n-2\sigma}{n+2\sigma}} =(-\Delta)^{-\sigma} v-a(t) v^{\frac{n-2\sigma}{n+2\sigma}} \quad \mbox{in }\R^n\times(0,\infty),
\]
where $v=|J_F|^{\frac{n-2\sigma}{2n}}(u\circ F)$.

Our second theorem is as follows.

\begin{thm}\label{thm:cc} The Cauchy problem \eqref{eq:main}-\eqref{eq:initial-data} has  a unique solution satisfying $u^{\frac{n-2\sigma}{n+2\sigma}}\in C^1([0,\infty);C^0(M))$. Moreover,
 as $t\to \infty$,
\[
a(t) \to a_\infty \quad \mbox{for some positive constant }a_\infty,
\]
and
\[
\int_{M} |Q_{K_0}^g-  a(t)|^q\,\ud vol_g \to 0\quad \mbox{for each }1\le q\le \frac{2n}{n-2\sigma}+\frac{n+2\sigma}{n-2\sigma}.
\]
\end{thm}

This theorem guarantees that the flow is a Palais-Smale flow-line of the associated critical functional; see Corollary \ref{cor:PS}.  Consequently, following Struwe \cite{St} one can show bubbling or global compactness if
\be\label{eq:pole-end}
\lim_{X\to Y} d_{g_0} (X,Y)^{n-2\sigma} K_0(X,Y)=c \quad \mbox{uniformly for }Y\in M, \tag{K-4}
\ee
where $\frac{1}{\Lda}\le c\le \Lda$ is a constant, and without loss of generality we assume $c=c_{n,\sigma}$.

Finally, we state a convergence result on $\Sn$.

\begin{thm}\label{thm:main1} Suppose that $M=\Sn$ and $\mathcal{K}_{g_0}= (P_{\sigma}^{g_0})^{-1}$ as in \eqref{eq:p-inverse}. If $u$ is a positive solution of \eqref{eq:main} on $\Sn\times (0,\infty)$ and $u(0)^{\frac{n-2\sigma}{n+2\sigma}}\in C^1(\Sn)$,  then $u\in C^1(\Sn \times (0,\infty))$ and
\[
\lim_{t\to \infty}\|u(\cdot,t) -  c \bar U_{\xi_0,\lda}( \cdot)\|_{C^0(M)} =0 \quad \mbox{on }\Sn\times [1,\infty),
\]
where $c>0$ is constant and $\bar U_{\xi_0,\lda}( \cdot)$ is the function in \eqref{eq:bubble}.
\end{thm}

A crucial step in the proof of the above theorem is establishing a differential Harnack inequality, which is inspired by Ye \cite{Ye}. However, in the current integral framework, we need to borrow ideas from Chen-Li-Ou \cite{CLO} and Li \cite{Li04}. The same differential Harnack inequality also holds on the family of locally conformally flat manifolds considered by Qing-Raske \cite{QR}. To prevent distractions, we exclusively focus on spherical objects.

The limiting case $2\sigma=n$ would involve a kernel having the rate of  $\ln d_{g_0}(X,Y)$, which deserves further exploration. In light of Dou-Zhu \cite{DZ}, it is intriguing to investigate the case when $2\sigma>n$. The above  differential-integral equations may be viewed as porous medium type equations with Riesz potentials diffusion. In contrast to the (fractional) Laplacian diffusion,
 \be\label{eq:strock-f}
\int_{M}|f|^{p-1} f\mathcal{K}_{g_0} (f)\,\ud vol_{g_0} \quad \mbox{may change signs, if }p>1.
 \ee
We further refer to  Bonforte-Endal \cite{BE} for the recent studies of  weak dual solutions of porous medium type equations, where Riesz potential estimates also play an important role.

\medskip

The paper is structured as follows. In Section \ref{s:SS}, we introduce a class of separable solutions that serve as a guiding principle for Theorem \ref{thm:main0}. These solutions are crucial in the proof of Theorem \ref{thm:main0}. In Section \ref{s:ER}, we demonstrate the existence of local solutions to \eqref{eq:IP} with $u(t)\in C(M)$ and provide a blow-up criterion that holds for all $m>0$. For future applications, we establish a regularity theorem for mild solutions, inspired by some idea of Li \cite{Li04}. In Section \ref{s:blt}, we establish crucial lower and upper bounds. Section \ref{s:creg} is dedicated to the critical case. In the first subsection, we prove Theorem \ref{thm:cc}. In the second one, we establish a differential Harnack inequality. In Section \ref{s:conver}, we prove a convergence theorem for all $m>0$ if the solutions are bounded between positive constants. The proofs of Theorem \ref{thm:main0} and Theorem \ref{thm:main1} are completed there.

\medskip

\textbf{Notations}:  Letters $x,y,z$  represent points in $\R^n$, and capital letters $X,Y,Z$ represent points on Riemannian manifolds.  Denote by $B_r(x)\subset \R^n$ the ball centered at $x$ with radius $r>0$. We may write $B_r$ in replace of $B_r(0)$ for brevity. For $X\in M$, $\B_{\delta}(X)$ denotes the geodesic ball centered at $X$ with radius $\delta$.
When it is clear from the context, we may use  $u(t)=u(\cdot, t)$ for instance.
Throughout the paper, constants $C>0$ in inequalities  may vary from line to line and are universal, which means they depend on given quantities but not on solutions.

\medskip

\noindent\textbf{Acknowledgement:} The author would like to express his gratitude to Tianling Jin at HKUST for the friendship and fruitful collaboration which has inspired some of the ideas in the current paper.

\section{Separable solutions}
\label{s:SS}

Let $K_0: M\times M\to (0,\infty]$ satisfy \eqref{eq:K-1}, \eqref{eq:K-2} and \eqref{eq:K-3}. We consider the functional
\be \label{eq:Jm}
J_m(f)= \frac{\int_{M} f \mathcal{K}_{g_0}(f)\,\ud vol_{g_0}}{(\int_{M} |f|^{m+1}\,\ud vol_{g_0})^{\frac{2}{m+1}}}
\ee and  the variational problem
\[
\bar J_m=\sup_{f\in L^{m+1}(M)\setminus \{ 0\}} J_m(f)>0.
\]
By the Hardy-Littlewood-Sobolev inequality, $\bar J_m<\infty$ if $m\ge \frac{n-2\sigma}{n+2\sigma}$.

\begin{prop}\label{prop:subcritical} If $m>\frac{n-2\sigma}{n+2\sigma}$, then $\bar J_m$ is achieved by a positive H\"older  continuous function.

\end{prop}

\begin{proof} Let $\{f_j\}_{j=1}^\infty$ be a maximizing sequence with $\|f_j\|_{L^{m+1}(M)}=1$
and
\be \label{eq:max-ass-1}
\lim_{j\to \infty}\int_{M} f_j \mathcal{K}_{g_0}(f_j)\,\ud vol_{g_0} = \bar J_m.
\ee
We may assume that $f_j$ are nonnegative and
\[
f_j \rightharpoonup f \quad \mbox{in }L^{m+1}(M).
\]
After passing to a subsequence (still denoted by $\{f_j\}$), we have
\[
\mathcal{K}_{g_0}(f_j) \to \mathcal{K}_{g_0}(f) \quad \mbox{in }L^q(M)
\]
for any $1\le q<\frac{n(m+1)}{(n-2\sigma(m+1))_+}$. See for instance the proof of Proposition 1.1 of \cite{HZ} or Proposition 5.1 of \cite{JX20}.  Since $m>\frac{n-2\sigma}{n+2\sigma}$, we have $\frac{n(m+1)}{(n-2\sigma(m+1))_+} >\frac{m+1}{m}$ with noting that $\frac{m+1}{m}$ is the H\"older conjugate exponent of $m+1$. Thus, by \eqref{eq:max-ass-1},
\[
\int_{M} f_j \mathcal{K}_{g_0}(f_j) \,\ud vol_{g_0}  \to \int_{M} f \mathcal{K}_{g_0}(f) \,\ud vol_{g_0}=\bar J_m .
\]
By the lower semi-continuity of $\|\cdot\|_{L^{m+1}(M)}$, we have $\|f\|_{L^{m+1}(M)}\le 1$ and thus
\[
J(f) \ge \bar J_m.
\]
Hence, $f$ is a maximizer and satisfies the Euler-Lagrange equation
\be\label{eq:steady}
\mathcal{K}_{g_0}(f)=\bar J_m f^{m} \quad \mbox{on }M.
\ee
By the standard potential estimates,  $f^m$ is H\"older continuous.
Since $f \ge 0$ but not identical to zero, $\mathcal{K}_{g_0}(f)$ must be positive everywhere on $M$. The proposition is proved.

\end{proof}

\begin{prop} \label{prop:unique}
If $m>1$, then  \eqref{eq:steady} has a unique positive continuous solution.
\end{prop}

\begin{proof} Suppose we have two positive continuous  solutions $S_1 $ and $S_2$.
Let
\[
\bar \al= \inf\{\al>0: S_1\le \al S_2 \quad \mbox{on }M\}.
\]
Then $S_1\le \bar \al S_2$ and $S_1$ is equal to $\bar  \al S_2$ at some points of $M$. Since the kernel $K_0(\cdot, \cdot)$ is positive,  it follows that either \[
\mathcal{K}_{g_0}(\bar \al S_2-S_1) > 0   \quad \mbox{on }M
\]
or
\[
 \bar \al S_2-S_1\equiv 0.
\]
If the former happens,
\[
(\bar \al S_2)^m-S_1^m =\bar  \al^{m-1}\mathcal{K}_{g_0}(\bar \al S_2)- \mathcal{K}_{g_0}(S_1) \ge \mathcal{K}_{g_0}(\bar \al S_2-S_1) >0 \quad \mbox{everywhere on }M.
\]
We obtain a contradiction.  Hence, $ \bar \al S_2-S_1\equiv 0$. By the equation \eqref{eq:steady}, we must have $\bar \al=1$. The proposition is proved.

\end{proof}

The following proposition extends a criterion of Aubin.

\begin{prop}\label{prop:critial} Suppose that  $K_0$ additionally satisfies \eqref{eq:pole-end}. If
\[
\bar J_{\frac{n-2\sigma}{n+2\sigma}}>S_{n,-\sigma},
\]
where $S_{n,-\sigma}$ is the best constant of Hardy-Littlewood-Sobolev inequality in \eqref{eq:lieb}, then $\bar J_{\frac{n-2\sigma}{n+2\sigma}}$ is achieved by a positive H\"older continuous function.
\end{prop}

A related result has been obtained by \cite{HZ}. Since \eqref{eq:pole-end} holds, one can show that $\bar J_{\frac{n-2\sigma}{n+2\sigma}}\ge S_{n,-\sigma}$, but the strict inequality needs some further conditions, such as ``positive mass" type condition:
\[
K_0(X,Y)=c_{n,\sigma} d_{g_0}(X,Y)+A(X,Y),
\]
where $A>0$ on $M$.

\begin{proof}[Proof of Proposition \ref{prop:critial}]  We shall adapt the blow-up analysis method proof in \cite{JX20} (see Proposition 5.3 there). Let $m_i>\frac{n-2\sigma}{n+2\sigma}$ and $\lim_{i\to \infty} m_i= \frac{n-2\sigma}{n+2\sigma}$.

First, we claim that
\be \label{eq:sup-greater}
\liminf_{i\to \infty} \bar J_{m_i}\ge \bar J_{\frac{n-2\sigma}{n+2\sigma}}.
\ee
Indeed, for any $\va>0$, we choose $\phi \ge 0$ such that
\[
\|\phi\|_{L^{\frac{2n}{n+2\sigma}}(M)}=1, \quad \int_{M} \phi\mathcal{K}_{g_0}(\phi) \,\ud vol_{g_0} \ge  \bar J_{\frac{n-2\sigma}{n+2\sigma}} -\va,
\]
and set
\[
\phi_i=\frac{\phi}{\|\phi\|_{L^{m_i+1}(M)}}.
\]
Then we have
\begin{align*}
\liminf_{i\to \infty} \bar J_{m_i} \ge \liminf_{i\to \infty}  J_{m_i}(\phi_i) \ge \liminf_{i\to \infty} \|\phi\|_{L^{m_i+1}(M)}^{-2} (J_{\frac{n-2\sigma}{n+2\sigma}} -\va).
\end{align*}
Sending $\va\to 0$,  \eqref{eq:sup-greater} follows immediately.

Upon passing to a subsequence, we may assume that $\lim_{i\to \infty}J_{m_i} = \Lda \ge  \bar J_{\frac{n-2\sigma}{n+2\sigma}}$.

Next, let $f_i$ be the positive maximizers obtained in Proposition \ref{prop:subcritical} for $J_{m_i}$, satisfying $\|f_i\|_{L^{m_i+1}}=1$ and \eqref{eq:steady} with $m=m_i$.

We claim that $\{f_i\}$ are uniformly bounded. By Riesz potential estimates, they will be uniformly bounded in some H\"older space. The proposition will be proved by extracting a subsequence.

If not, there is a subsequence, still denoted by $\{f_i\}$, which satisfies
\[
f(X_i):=\max_{M} f_i \to \infty \quad \mbox{as }i\to \infty,
\]
where $X_i\in M$. Choose a geodesic normal coordinates system $\{x^1,\dots, x^n \}$ centered at $X_i$, and write the integral equation as
\[
\bar J_{m_i} f_i (\exp_{X_i} x ) ^{m_i}= \int_{B_{\delta}} K_0 (\exp_{X_i} x,  \exp_{X_i} y)\sqrt{\det  g_0}  f_i (\exp_{X_i} y)\,\ud y+ h_{i}(x),
\]
where $h_i(x)= \int_{M\setminus \mathcal{B}_{\delta}(X_i)} K_0(\exp_{X_i} x ,\zeta) f_i(\zeta)\,\ud vol_{g_0} (\zeta)$ and $\delta>0$ is a small constant. Since $K_0$ additionally satisfies \eqref{eq:pole-end},
following the blow up analysis procedure in the proof of Proposition 2.11 in \cite{JLX17} or Proposition 4.1 in \cite{LX19}, we have, after passing to a subsequence, as $i\to \infty$,
\[
v_i(x) :=\frac{1}{f_i(X_i) } f_i (\exp_{X_i} f_i(X_i)^{-\frac{1-m_i}{2\sigma}} x ) \to v(x) \quad \mbox{in }C_{loc}^0 (\R^n),
\]
for some $v>0$ satisfying
\[
\Lda v(x)^{\frac{n-2\sigma}{n+2\sigma}}=c_{n,\sigma} \int_{\R^n} \frac{v(y)}{|x-y|^{n-2\sigma}}\,\ud y \quad x\in \R^n.
\]
In fact, $v$ is classified in Chen-Li-Ou \cite{CLO}.
It follows that
\[
\Lda \int_{\R^n} v^{\frac{2n}{n+2\sigma}} = c_{n,\sigma} \int_{\R^n\times \R^n} \frac{v(x)v(y)}{|x-y|^{n-2\sigma}}\,\ud y\ud x\le S_{n,-\sigma}  \left(\int_{\R^n} v^{\frac{2n}{n+2\sigma}}\,\ud x\right)^{\frac{n+2\sigma}{n}}.
\]
Since $\|v_i\|_{L^{m_i+1}(B_{\delta \cdot f_i(X_i)^{\frac{1-m_i}{2\sigma}} })}\le \|f_i\|_{L^{m_i+1}(M)}=1$, $\int_{\R^n} v^{\frac{2n}{n+2\sigma}}\le 1$. It follows that
\[
\bar J_{\frac{n-2\sigma}{n+2\sigma}}\le \Lda \le S_{n,-\sigma} ,
\]
which contradicts to the assumption of the proposition. Hence, the claim is verified and  the proposition is proved.
\end{proof}

If $m\neq 1$, then $S=\bar J_m^{\frac{1}{m-1}} f$ is a solution of \eqref{eq:st0}.

Suppose that $S$ is a continuous positive solution of  \eqref{eq:st0} and  $h(t) S$ is a positive separable solution of
\be\label{eq:pme}
\pa_t u^m= \mathcal{K}_{g_0}(u) \quad \mbox{on }M\times (0,T).
\ee
Then $h$ must satisfy the ODE
\[
\frac{\ud h^m}{\ud t}=h  \quad \mbox{on }(0,T).
\]
Integrating the above equation yields
\[
h=h_c:= (c+\frac{m-1}{m} t)^{\frac{1}{m-1}},  \quad  \quad c\ge 0,
\]
where $c>0$ if $m<1$.
Hence, we obtain the separable solutions
\be \label{eq:separable}
U_c(X,t)= (c+\frac{m-1}{m} t)^{\frac{1}{m-1}} S(X).
\ee
It is worth noting that if $m>1$, $U_0(0)\equiv 0$ and hence $U_0$ corresponds to the ``friendly giant" of porous medium equations in bounded domains. This solution originates from the loss of uniqueness of the ODE.

\section{Existence and regularity}

\label{s:ER}

By the standard estimates for Riesz potentials, we have
\be \label{eq:riesz-bound-1}
\|\mathcal{K}_{g_0}(f) \|_{L^{\frac{np}{n-2\sigma p }} (M)}\le C_1 \|f\|_{L^p(M)},  \quad \forall~ f\in L^p (M), 1<p<\frac{n}{2\sigma},
\ee
\be \label{eq:riesz-bound-2}
\|\mathcal{K}_{g_0}(f) \|_{C^{\alpha }(M)} \le C_1\|f\|_{L^p(M)}, \quad \forall~ f\in L^p (M), \frac{n}{2\sigma} <p\le \infty,
\ee
where $0<\alpha<\min\{2\sigma,1\} $ and  $C_1>0$ depends only on $M,g_0,p, \Lda $ and $\sigma$.

\begin{lem}\label{lem:existence} Suppose that $u_0\in C^0(M)$. Then \eqref{eq:IP} has a unique solution $u$ satisfying  $u^m\in C^1([0,T^*); C^0(M))$ for some $T^*>0$ representing the maximum existence time. If $T^*<\infty$, then
\[
\lim_{t\to T^*}\|u(t)\|_{C^0(M)}=\infty.
\]
\end{lem}

\begin{proof}  Integrating \eqref{eq:IP} in $t$, we obtain
\be \label{eq:picard}
u(X,t)^{m}= u_0(X)^{m}+ \int_0^t \mathcal{K}_{g_0}(u) (X,s)\,\ud s.
\ee
The existence follows from Picard-Lindel\"of theorem, using \eqref{eq:riesz-bound-2}. Note that
\[
\pa_tu(X,t)^{m} \Big|_{t=0}= \mathcal{K}_{g_0}( u_0)>0 \quad \mbox{on }\Sn,
\]
we conclude that  $u^{m}$ is positive for $t>0$ and increasing in $t$. If $\lim_{t\to T^*}\|u(t)\|_{C^0(M)}$ is finite, we can extend the solution further, which contradicts to the definition of $T^*$.  The lemma is proved.
\end{proof}

We may say that a nonnegative function  $u\in C((0,T); L^p(M))$, $p\ge \max\{m,1\}$,  is  a  mild solution of \eqref{eq:IP} if the  integral identity  \eqref{eq:picard} holds for almost every $X\in M$ and $t\in [0,T)$.

Let us write the equation in local coordinates. For an arbitrary point $\bar X\in M$,  choose a geodesic normal coordinates system $\{x^1,\dots, x^n \}$ centered at $\bar X$, and write the integral equation \eqref{eq:picard} as
\be \label{eq:local-coord}
\hat u(x,t)^{m}= \hat u(x,0)^{m}+ \int_0^t \int_{{B}_{\delta }} \hat K_0( x, y) \hat u( y,s)\,\ud y\ud s+h(x,t),
\ee
where $\delta>0$ is a constant,
\[
\hat u(x,t)= u(\exp_{\bar X} x,t), \quad \hat K_0(x,y)= K_0(\exp_{\bar X} x,\exp_{\bar X} y) \sqrt{\det g_0(y)}
\] and  \[
h(x,t)= \int_0^t \int_{M\setminus \mathcal{B}_{\delta}(\bar X)} K_0(\exp_{\bar X} x ,\zeta) u(\zeta,s)\,\ud vol_{g_0} \ud s.
 \]
In the next lemma, we show that $h$ is a good term.

\begin{lem} \label{lem:tail} Assume as above. Then the nonnegative function $h$ satisfies that, for every $0<t<T^* $,
\be\label{eq:h-harnack}
\sup_{B_{\delta/2}} h(\cdot, t)\le  C \inf _{B_{\delta/2}} h(\cdot, t),
\ee
\be\label{eq:h-infty}
\sup_{B_{\delta/2}} h(\cdot, t)\le  C \dashint_{\mathcal{B}_{\delta/2}(\bar X))} u^m \,\ud vol_{g_0}
\ee
and
\be\label{eq:h-grad}
|\nabla h (x,t)|\le \frac{ C}{\delta} h(x,t), \quad \forall~ x\in B_{\delta/2},
\ee
where $\dashint_{\mathcal{B}_{\delta/2}(\bar X))}= \frac{1}{vol_{g_0} (\mathcal{B}_{\delta/2}(\bar X)))} \dashint_{\mathcal{B}_{\delta/2}(\bar X))}$, and $C>0$ depends only on $n, \sigma, \lda $ and $\|g_0\|_{C(M)}$.
\end{lem}
\begin{proof}
Note that for any $x_1,x_2\in B_{\delta/2}$, using \eqref{eq:K-2},
\begin{align*}
h(x_1,t)&=  \int_0^t \int_{M\setminus \mathcal{B}_{\delta}(\bar X)} \frac{K_0(\exp_{\bar X} x_1 ,\zeta)}{K_0(\exp_{\bar X} x_2 ,\zeta)} K_0(\exp_{\bar X} x_2 ,\zeta) u(\zeta,s)\,\ud vol_{g_0} \ud s\\&
\le \Lda^2 \int_0^t \int_{M\setminus \mathcal{B}_{\delta}(\bar X)} \Big(\frac{d_{g_0}(x_1,\zeta)}{d_{g_0}(x_2,\zeta)}\Big)^{2\sigma-n} K_0(\exp_{\bar X} x_2 ,\zeta) u(\zeta,s)\,\ud vol_{g_0} \ud s\\&
\le C h(x_2, t).
\end{align*}
Hence, \eqref{eq:h-harnack} is verified.  Since $h(x,t)\le \hat u(x,t)^m$, \eqref{eq:h-infty} follows immediately from \eqref{eq:h-harnack}.

Using \eqref{eq:K-3} and \eqref{eq:K-2}, we have
\begin{align*}
|\nabla h(x,t)|\le  \frac{\Lda ^2 C}{\delta} \int_0^t \int_{M\setminus \mathcal{B}_{\delta}(\bar X)} K_0(\exp_{\bar X} x,\zeta) u(\zeta,s)\,\ud vol_{g_0} \ud s.
\end{align*}
Hence, \eqref{eq:h-grad} is verified. The lemma is proved.
\end{proof}

We shall establish  regularity results for mild solutions, which are inspired by Theorem 1.3 of Li \cite{Li04}.
For $T>0$, suppose that  $V\in L^\infty([0,T]; L^\frac{n}{2\sigma}(B_3))$ and $h\in L^\infty([0,T];L^q(B_2))$, $q> \frac{n}{n-2\sigma}$, are nonnegative functions.
We study the integrability improvement for nonnegative solutions of the integral inequality
\be\label{eq:local-1}
w(x, t)\le   \int_0^t e^{s-t} \int_{B_3}\frac{V(y,s) w(y,s)}{|x-y|^{n-2\sigma}}\,\ud y\ud s +h(x,t), \quad a.e.~ x\in B_2, ~0\le t\le T.
\ee
Suppose  $ w\in L^\infty((0,T); L^p(B_3))$ for some $\frac{n}{n-2\sigma}<p<q$.

The  factor $e^{s-t}$ will serve to establish estimates independent of $T$. In subsequent applications, it will be substituted with $e^{\alpha (s-t)}$ for a positive constant $\alpha$. For brevity, we set $\alpha=1$ in this context.

\begin{thm} \label{thm:improve-int}
 Assume as above. Suppose additionally that $V\in C([0,T]; L^\frac{n}{2\sigma}(B_3))$.  Then $ w\in L^\infty((0,T); L^q(B_1))$.
\end{thm}

First, we prove the following proposition.

\begin{prop}\label{prop:small_to_regular} For $q>p>\frac{n}{n-2\sigma}$, there exist positive constants $\bar \delta<1$ and $C\ge 1$, depending only on $n,\sigma$, $p$ and $q$, such that if
\be \label{eq:smallness-1}
\|V\|_{L^\infty((0,T); L^{\frac{n}{2\sigma}}(B_3))} \le \bar \delta,
\ee then $ w\in L^\infty((0,T); L^q(B_1))$ and
\[
\|w\|_{L^\infty((0,T); L^q(B_1))}\le C\big(\|w\|_{L^\infty((0,T); L^p(B_3))} +\|h\|_{L^\infty((0,T); L^q(B_2))}\big).
\]
\end{prop}

\begin{proof}
We may assume \textit{a priori} that $w\in L^\infty((0,T); L^q(B_2))$ and only prove the estimate. Otherwise, one can truncate the kernel and take an approximation  as  implemented by Li \cite{Li04}.

For any open set $\om \subset B_3$, we let
\[
D_{\om}(x,s)= \int_{\om}\frac{V(y,s) w(y,s)}{|x-y|^{n-2\sigma}}\,\ud y,  \quad x\in B_3,
\]
and let
\begin{align*}
I_{1,r}(x,t)&=\int_0^t e^{s-t} D_{B_r}(x,s)\ud s, \\
 I_{2,r}(x, t)&=\int_0^t e^{s-t} D_{B_3\setminus \bar B_r}(x,s)\,\ud s
\end{align*}
with $0<r <3/2$.
For any fixed $t$, by the Minkowski inequality and estimates of Riesz potential we have, for $0<\rho<r <3/2$,
\begin{align*}
\|I_{1,r}(\cdot,t)\|_{L^q(B_{\rho})} &\le \int_0^t e^{s-t} \left(\int_{B_\rho} D_{B_r}(x,s)^q\,\ud x\right)^{1/q}\,\ud s\\&
\le C \int_0^t e^{s-t} \|V(s)^\nu w(s)^\nu \|_{L^1(B_r)}^{1/\nu}\,\ud s \\&
\le C  \int_0^t e^{s-t} ( \|V(s)^\nu\| _{L^{\frac{q}{q-\nu}}(B_r)}  \| w(s)^\nu \|_{L^{\frac{q}{\nu}}(B_r)})^{1/\nu}\,\ud s\\&
= C  \int_0^t e^{s-t} \|V(s)\| _{L^{\frac{n}{2\sigma}}(B_r)}  \| w(s) \|_{L^{q}(B_r)}\,\ud s\\&
\le C \bar \delta   \| w \|_{L^\infty ((0,T); L^{q}(B_r))} \int_0^{T}e^{s-T}\,\ud s\\&
\le C \bar \delta   \| w \|_{L^\infty ((0,T); L^{q}(B_r))} \\&
\le \frac{1}{2}  \| w \|_{L^\infty ((0,T); L^{q}(B_r))} ,
\end{align*}
where $\frac{1}{q}=\frac{1}{\nu}- \frac{2\sigma}{n}$, if $C\bar \delta\le 1/2$.
Using, for $x\in B_{\rho}$,
\begin{align*}
|D_{B_3\setminus \bar B_r}(x,s) |&\le \frac{1}{(r-\rho)^{n-2\sigma}} \int_{B_3\setminus \bar B_r} V(y,s) w(y,s)\,\ud y \\&
\le   \frac{C}{(r-\rho)^{n-2\sigma}} \|V\|_{L^{\frac{n}{2\sigma}}(B_3)} \|w\|_{L^{p}(B_3)},
\end{align*}
we have
\begin{align*}
\|I_{2,r}(\cdot,t)\|_{L^q(B_{\rho})}  \le  \frac{C}{(r-\rho)^{n-2\sigma}} \|w\|_{L^{p}(B_3)}.
\end{align*}
Since $w\ge 0$,
\begin{align*}
\|w\|_{L^\infty((0,T); L^q(B_\rho))}\le& \frac{1}{2} \|w\|_{L^\infty((0,T); L^q(B_r))}\\&
 +   \frac{C}{(r-\rho)^{n-2\sigma}}  \|w\|_{L^{p}(B_3)}  +C\|h\|_{L^\infty((0,T); L^q(B_2))}.
\end{align*}
Using the Lemma 1.1 from Giaquinta-Giusti \cite{GG}, we are able to establish the desired estimate.
\end{proof}

\begin{proof}[Proof of Theorem \ref{thm:improve-int}]  Let  $x_0\in B_1$ and  $0<\va<1/4$ be  small, we let
\[
w_\va(x,t)= \va ^{\frac{n-2\sigma}{2}} w(x_0+ \va x, t), \quad  V_\va(x,t)= \va^{2\sigma} V(x_0+ \va x, t), \quad x\in B_3,
\]
and
\[
h_\va(x,t)= \va ^{\frac{n-2\sigma}{2}} \int_0^t e^{s-t}  \int_{B_3\setminus B_{3\va} (x_0)} \frac{V(y,s)w(y,s)}{ |x_0+\va x-y|^{n-2\sigma}}\,\ud y \ud s+ \va ^{\frac{n-2\sigma}{2}} h(x_0+ \va x, t).
\]
Since $V\in C([0,T]; L^\frac{n}{2\sigma}(B_3))$ and
\[
\| V_\va(\cdot,t)\|_{L^{\frac{n}{2\sigma}} (B_3)}= \| V(\cdot,t)\|_{L^{\frac{n}{2\sigma}} (B_{3\va}(x_0))},
\]
we can find a small $\va$ such that  $\| V_\va\|_{L^\infty((0,T); L^{\frac{n}{2\sigma}} (B_3))}\le \bar \delta$ as in the above proposition. The theorem follows from Proposition \ref{prop:small_to_regular}.
\end{proof}

\section{Bounds at large time}
\label{s:blt}

In this section, we establish sharp asymptotical behavior of solutions of \eqref{eq:IP} if $m>1$; and obtain sharp integral bounds if $\frac{n-2\sigma}{n+2\sigma}\le m<1$.

First, we need a comparison principle.

\begin{lem}[Comparison principle] \label{lem:cp-1} Suppose that $m>0$.  Suppose that $f_1,f_2\in C^1([0,T]; C^0(M))$ are nonnegative functions satisfying
\[
\pa_t f_1^m \ge \mathcal{K}_{g_0} (f_1), \quad \pa_t f_2^m \le \mathcal{K}_{g_0} (f_2) \quad \mbox{on }M\times (0,T]
\]
and
\[
f_1(0)\ge f_2(0) \quad \mbox{on }M, \quad \mbox{and }f_1(0)>f_2(0)\quad \mbox{somewhere}.
\]
Then $f_1>f_2$ on $M\times (0,T]$.
\end{lem}

\begin{proof} This proof is quite elementary. We omit it.

\end{proof}

\begin{prop}\label{prop:mge1} Assume as in Lemma \ref{lem:existence}. If $m>1$, then $T^*=\infty$ and
\[
\|\frac{u}{U_0}-1\|_{C(M)}\le \frac{C}{t} \quad \mbox{for }t>1.
\]
\end{prop}

\begin{proof} Let $U_c$ be the separable solution defined in \eqref{eq:separable}. Let $c$  be large so that
\[
0\equiv U_0(0)\le u_0< U_c(0).
\]
By Lemma \ref{lem:cp-1}, we have
\[
U_0\le u\le U_c \quad \mbox{in }M\times (0,T^*).
\]
Since $U_c$ is locally uniformly bounded, by Lemma \ref{lem:existence} we must have $T^*=\infty$. The two sides bounds also imply that
\[
0\le \frac{u}{U_0}-1 \le \frac{U_c}{U_0}-1= (\frac{\frac{c}{t} +\frac{m-1}{m}}{\frac{m-1}{m}})^{\frac{1}{m-1}}-1=O(\frac{1}{t}).
\]
The proposition is proved.

\end{proof}

When $m=\frac{n-2\sigma}{n+2\sigma}$, we will need to use the modulus
\[
\w_t(\rho):=\sup_{X\in M}\int_{\mathcal{B}_\rho(X)} u^{m+1}(t)\,\ud vol_{g_0},\quad  \rho>0.
  \]

\begin{lem}\label{lem:lp-l-infty} Suppose that $\frac{n-2\sigma}{n+2\sigma}\le  m<1$ and
\[
 \|u(T)\|_{ L^{m+1}(M)}\le \bar C
\]
for some $0<T\le T^*$. Then
\[
\max_{[0,T]} \|u\|_{C^{0}(M)}\le C,
\]
where  $C>0$ depends only on $M,\sigma, \Lda, m, T, \bar C$ and $\|u_0\|_{C^0(M)}$, and further on the modulus $\w_T(\cdot)$
when $m=\frac{n-2\sigma}{n+2\sigma}$.
\end{lem}

\begin{proof} Let
\[
\tilde V(X,t)= u(X,t)^{1-m}.
\]
Note that  $u(X,t)$ is increasing in $t$ for any fixed $X$. We have   $\w_t(\rho)\le \w_T(\rho)$ for $t\le T$. Moreover,
if $m>\frac{n-2\sigma}{n+2\sigma}$, for any $\bar X\in M$ and $\rho>0$, by the H\"older inequality
\be \label{eq:sub-small}
\begin{split}
&\int_{\mathcal{B}_\rho(\bar X)} |\tilde V(t)|^{\frac{n}{2\sigma}}\,\ud vol_{g_0}\le
\int_{\mathcal{B}_\rho(\bar X)} |\tilde V(T)|^{\frac{n}{2\sigma}}\,\ud vol_{g_0}\\&
\le \left(\int_{\mathcal{B}_\rho(\bar X)} |u(T)|^{m+1}\,\ud vol_{g_0} \right)^{\frac{n(1-m)}{2\sigma (1+1)}} \left(vol_{g_0}(\mathcal{B}_\rho(\bar X) )\right)^{\frac{2\sigma(m+1)}{(n+2\sigma) m-(n-2\sigma)}}\\ &
\le C \bar C ^{\frac{n(1-m)}{2\sigma (1+1)}} \rho^{\frac{2n\sigma(m+1)}{(n+2\sigma) m-(n-2\sigma)}},
\end{split}
\ee
which can be made to be small by choosing small $\rho$ but independent of modulus $\w_T(\cdot)$.   Then the lemma is a consequence of Lemma \ref{lem:tail}, Proposition \ref{prop:small_to_regular}, and a bootstrap argument.

%Below we give a direct proof.

%For any $0< t\le T$, by the equation we have
%\[
%u^m(t) =u_0^m + \int_0^t \mathcal{K}_{g_0}(u)\,\ud s.
%\]

%It follows that, for any $q>1 $,
%\begin{align*}
%\|u^m(t)\|_{L^q(M)}\le \|u_0^m\|_{C^{0}(M)}+ t^{\frac{1}{q'}}\left( \int_0^t \int_M |\mathcal{K}_{g_0}(u)|^q \right)^{1/q},
%\end{align*}
%where $\frac{1}{q}+\frac{1}{q'}=1$.
%If $m>\frac{n-2\sigma}{n+2\sigma}$, by \eqref{eq:riesz-bound-1},
%\[
%\int_M |\mathcal{K}_{g_0}(u(\cdot, s))|^q \le C \|u(\cdot,s)\|_{L^{m+1}(M)}^q, \quad 0<s<T,
%\]
%for any $\frac{m+1}{m}\le q<\frac{n(m+1)}{(n-2\sigma(m+1))_+}$. It follows that
%\[
%\max_{[0,T]} \|u\|_{L^{q}(M)}<\infty.
%\]
%By iterating in a finite times, the $C^0$ bound will be included.

%If $m=\frac{n-2\sigma}{n+2\sigma}$, this lemma follows from Theorem \ref{thm:improve-int} and  a bootstrap argument.
\end{proof}

\begin{cor}\label{cor:lm-bl} If $\frac{n-2\sigma}{n+2\sigma}\le m<1$, then $T^*<\infty$ and
\[
\lim_{t\to T^*} \int_M u^{m+1}(t)\,\ud vol_{g_0} =\infty.
\]
\end{cor}

\begin{proof}
For a small $t_0>0$, choose a large $c$ such that
\[
U_c(X,t_0)=(c-\frac{1-m}{m} t_0)^{-\frac{1}{1-m}} S < u(X,t_0) \quad \mbox{on } M.
\]
By the  comparison principle, we have
\[
u\ge U_c \quad \mbox{in }M\times [t_0,T^*).
\]
Hence, $T^*\le \frac{cm}{1-m}<\infty$.

As $u(X,t)$ is increasing in $t$ for any fixed $X\in M$, the limit in the lemma exists. If the limit is finite, by the monotone convergence theorem we have
\[
\lim _{t\to T^*} u(\cdot ,t)=  u(\cdot,T^*) \in L^{m+1}(M).
\]
Using Lemma \ref{lem:lp-l-infty}, we obtain
\[
\max_{[0,T^*]} \|u\|_{C^{0}(M)}\le C,
\]
which contradicts to Lemma \ref{lem:existence}. Therefore, the corollary holds.

\end{proof}

\begin{prop}\label{prop:fde-1} If $\frac{n-2\sigma}{n+2\sigma}\le m<1$, then
\[
\frac{1}{C} (T^*-t)^{-\frac{1}{1-m}}\le \|u(t)\|_{L^{m+1}(M)}\le C (T^*-t)^{-\frac{1}{1-m}},
\]
where $C>0$ is a constant.
\end{prop}

\begin{proof}
By a direct computation using the first equation in \eqref{eq:IP}, we obtain
\begin{align*}
\frac{\ud }{\ud t}\int_{M}u^{m+1}\,\ud vol_{g_0} =\frac{m+1}{m}\int_{M} u \mathcal{K}_{g_0} (u) \,\ud vol_{g_0}  \ge 0,
\end{align*}
\begin{align*}
\frac{\ud }{\ud t}\int_{M} u \mathcal{K}_{g_0} (u)\,\ud vol_{g_0}  &=2 \int_{M} \pa_t u \mathcal{K}_{g_0} (u)\,\ud vol_{g_0} \\&
=\frac{2}{m}\int_{M} | \mathcal{K}_{g_0} (u)|^2 u^{1-m}\,\ud vol_{g_0} .
\end{align*}
We may drop $ \ud vol_{g_0}$ in the follow integrals. Hence,
\begin{align*}
\frac{\ud J_m(u)}{\ud t}= \frac{2}{m}\Big(\int_M u^{m+1}\Big)^{-\frac{2}{m+2}}\left[\int_{M} | \mathcal{K}_{g_0} (u)|^2 u^{1-m} -\frac{(\int_{M} u \mathcal{K}_{g_0} (u))^2}{\int_{M}u^{m+1}} \right].
\end{align*}
By the H\"older inequality,
\begin{align*}
\left(\int_{M} u \mathcal{K}_{g_0} (u)\right)^2&=\left(\int_{M} u^{\frac{1+m}{2}} \big(u^{\frac{1-m}{2}}\mathcal{K}_{g_0} (u)\big) \right)^2\\&
\le \left(\int_{M} | \mathcal{K}_{g_0} (u)|^2 u^{1-m} \right) \int_M u^{m+1}.
\end{align*}
It follows that
\[
\frac{\ud J_m(u)}{\ud t} \ge 0.
\]

Set
\[
Z(t)= \left(\int_{M} u^{m+1}\right)^{-\frac{2}{m+1}+1} .
\]
Then
\begin{align*}
Z'(t)= \frac{m-1}{m} J_m(u)\ge -\frac{1-m}{m} \bar J_m.
\end{align*}
By integration, for $0<t<T<T^*$,
\[
Z(t)\le \frac{1-m}{m} \bar J_m (T-t) +Z(T).
\]
Sending $T$ to $T^*$, by Corollary \ref{cor:lm-bl} we have
\[
Z(t) \le \frac{1-m}{m} \bar J_m \cdot (T^*-t) .
\]
On the other hand,
\[
Z''(t)= \frac{m-1}{m}\frac{\ud }{\ud t} J_m(u)\le 0.
\]
For any $0<s<t<T<T^*$, we have
\[
Z(t)\ge Z(T)+\frac{Z(s)-Z(T)}{s-T}(t-T).
\]
Sending $s\to 0$ and $T\to T^*$, by Corollary \ref{cor:lm-bl} we obtain
\[
Z(t) \ge \frac{Z(0)}{T^*}(T^*-t).
\]
Therefore, the proposition is proved.
\end{proof}

In order to study the blow up profile of $u$ near $T^*$, let us introduce a re-normalization of $u$ as follows.  For $0<m<1$, let
\be \label{eq:renorm}
\tilde u(X,\tau)=(T^*-t)^{\frac{1}{1-m}} u(X,t), \quad \tau= -\ln\frac{T^*-t}{T^*}.
\ee
Then we have
\begin{equation} \label{eq:main-rescaled}
\pa_\tau \tilde u ^m=  \mathcal{K}_{g_0}(\tilde u ) -\frac{m}{1-m} \tilde u^m \quad \mbox{in }M\times (0,\infty).
\end{equation}
From Proposition \ref{prop:fde-1}, we know that
\[
\frac{1}{C^{m+1}}\le \int_{M} \tilde u (\tau)^{m+1}\,\ud vol_{g_0} \le C^{m+1} , \quad \tau\in [0,\infty).
\]

\begin{prop}\label{prop:fde-2} If $\frac{n-2\sigma}{n+2\sigma}< m<1$, then
\[
\frac{1}{C}\le \|\tilde u(\tau)\|_{L^\infty(M)}\le C , \quad \tau\in [1,\infty)
\]
and thus
\[
\frac{1}{C} (T^*-t)^{-\frac{1}{1-m}}\le \|u(t)\|_{L^\infty(M)}\le C (T^*-t)^{-\frac{1}{1-m}}, \quad T^*(1-e^{-1}) \le t< T^*,
\]
where $C>0$ is a constant.
\end{prop}

\begin{proof} By integrating \eqref{eq:main-rescaled}, we have
\[
\tilde u(X, \tau)^m= e^{-\frac{m}{1-m} \tau} \tilde u(X, 0)^m +\int_0^{\tau} e^{\frac{m}{1-m}(s-\tau)}  \mathcal{K}_{g_0}(\tilde u )(X,s)\,\ud s.
\]
Since $m>\frac{n-2\sigma}{n+2\sigma}$, we have \eqref{eq:sub-small}. Making use of Lemma \ref{lem:tail}, Proposition \ref{prop:small_to_regular} and a bootstrap argument, we then can obtain
\[
\tilde u(X, \tau) \le C_1
\]
for some $C_1>1$.

On the other hand,
\[
\int_M \tilde u \ge \frac{1}{C_1^{m}} \int_M \tilde u ^{m+1}\ge \frac{1}{C C_1^m},
\]
and, using \eqref{eq:K-2},
\[
 \mathcal{K}_{g_0}(\tilde u )(X,s) \ge \frac{1}{C}
\]
for some $C$ independent of $s$.  It follows that
\[
\tilde u(X, \tau)^m \ge \frac{1}{C} \int_0^{\tau} e^{\frac{m}{1-m}(s-\tau)}\,\ud s\ge \frac{1}{C} \frac{1-m}{m} (1-e^{-\frac{m}{1-m}}),
\]
if $\tau \ge 1$. Hence, the first conclusion is verified. The second one then follows from the definition of $\tilde u$. Therefore,  the proposition is proved.
\end{proof}

\section{The critical regime}
\label{s:creg}

In this section, we set $m=\frac{n-2\sigma}{n+2\sigma}$.

We may further normalize $\tilde u$ in \eqref{eq:main-rescaled} as
\[
 \tilde u/\|\tilde u\|_{L^{m+1}(M)},
\]
which turns out to be a solution of the normalized equation \eqref{eq:main} on $M\times (0,\infty)$.

On the other hand, by changing variables Lemma \ref{lem:existence} implies that \eqref{eq:main} and \eqref{eq:initial-data}  admits a unique positive solution $u$ satisfying  $u^m\in C^1([0,T^*); C^0(M))$, where $0<T^*\le \infty$ is taken to be the maximal existence time of the solution.

\subsection{Long time existence and concentration compactness}

Let $u$ be a positive  solution of \eqref{eq:main} and $g=u^{\frac{4}{n+2\sigma}} g_0$.  Set
\[
V(t)= \int_{\Sn} u(t)^{\frac{2n}{n+2\sigma}} \,\ud vol_{g_0}
\]
and
\[
M_q(t)= \int _{\Sn} | Q_{K_0}^g -a(t)|^q \,\ud vol_g= \int _{\Sn} | Q_{K_0}^g -a(t)|^q u^{\frac{2n}{n+2\sigma}} \,\ud vol_{g_0} , \quad  q\ge 1,
\]
where $Q_{K_0}^g=\mathcal{K}_{g}(1)$ as defined in \eqref{eq:q-k}.

\begin{lem}\label{lem:properties} Along the flow, we have
\begin{itemize}
\item[(i)]
\begin{align*}
\frac{\pa }{\pa t} u ^{\frac{2n}{n+2\sigma}}= \frac{2n}{n-2\sigma}  (Q_{K_0}^g -a(t))u ^{\frac{2n}{n+2\sigma}}
\end{align*}
and thus $V'(t)=0$;
\item[(ii)]
\begin{align*}
\frac{\pa }{\pa t}   (Q_{K_0}^g -a(t))=\frac{n+2\sigma}{n-2\sigma }  \mathcal{K}_g (Q_{K_0}^g-a) -(Q_{K_0}^g-a)^2 - a(Q_{K_0}^g-a) -a';
\end{align*}
\item[(iii)]
\[
\frac{\ud }{\ud t} J_{\frac{n-2\sigma}{n+2\sigma}}(u)= \frac{2(n+2\sigma)}{n-2\sigma}  V(t)^{-\frac{n+2\sigma}{n}} \cdot M_2(t) \ge 0,
\]
where $J_{\frac{n-2\sigma}{n+2\sigma}}$ is as defined in \eqref{eq:Jm}.
\end{itemize}
\end{lem}

\begin{proof} By \eqref{eq:main}, we have
\be\label{eq:log-u}
\frac{\pa_t u}{u}= \frac{n+2\sigma}{n-2\sigma} (Q_{K_0}^g -a(t)).
\ee
The first item follows. Using the definition of $Q_{K_0}^g$ and \eqref{eq:log-u}, we have
\begin{align*}
\frac{\pa }{\pa t}   Q_{K_0}^g &= \frac{\pa }{\pa t}   (u^{-\frac{n-2\sigma }{n+2\sigma }} \mathcal{K}_{g_0} (u))\\&
= -\frac{n-2\sigma}{n+2\sigma}\frac{\pa_t u}{u}  Q_{K_0}^g  + u^{-\frac{n-2\sigma }{n+2\sigma }} \mathcal{K}_{g_0} (\pa_t u)\\&
= -(Q_{K_0}^g-a)Q_{K_0}^g +\frac{n+2\sigma}{n-2\sigma }  u^{-\frac{n-2\sigma }{n+2\sigma }} \mathcal{K}_{g_0} (u (Q_{K_0}^g-a))\\&
=  -(Q_{K_0}^g-a)Q_{K_0}^g +\frac{n+2\sigma}{n-2\sigma }  \mathcal{K}_g (Q_{K_0}^g-a)\\&
=-(Q_{K_0}^g-a)^2 - a(Q_{K_0}^g-a) +\frac{n+2\sigma}{n-2\sigma }  \mathcal{K}_g (Q_{K_0}^g-a).
\end{align*}
Finally,
 \begin{align*}
 \frac{\ud }{\ud t} J_{\frac{n-2\sigma}{n+2\sigma}}(u) &= \frac{2(n+2\sigma)}{n-2\sigma}  V(t)^{-\frac{n+2\sigma}{n}} \cdot \Big[\int_{\Sn} |\mathcal{K}_{g_0}(u)|^2 u^{\frac{4\sigma}{n+2\sigma}}  - a(t) u \mathcal{K}_{g_0} (u) \,\ud vol_{g_0}  \Big ] \\&
 =  \frac{2(n+2\sigma)}{n-2\sigma}  V(t)^{-\frac{n+2\sigma}{n}} \cdot  \int |\mathcal{K}_{g_0}(u)-  a(t) u^{\frac{n-2\sigma}{n+2\sigma}}  |^2 u^{\frac{4\sigma}{n+2\sigma}} \,\ud vol_{g_0}   \\&
 = \frac{2(n+2\sigma)}{n-2\sigma}  V(t)^{-\frac{n+2\sigma}{n}} \cdot M_2(t).
 \end{align*}
The lemma is proved.

\end{proof}

Without loss of generality, we assume from now on
\be\label{eq:unit-vol}
V(t)=1.
\ee

\begin{lem}\label{lem:normalized contant} We have
\[
 a'(t) =\frac{2(n+2\sigma)}{n-2\sigma}  M_2(t) \ge0
\]
and
\[
0<J_{\frac{n-2\sigma}{n+2\sigma}}(u_0)\le a(t)\le \bar J_{\frac{n-2\sigma}{n+2\sigma}} .
\]
Hence, $\lim_{t\to \infty} a(t)=:a_\infty$ exists.
\end{lem}

\begin{proof} Since $V(t)=1$,
\[
a(t)= J_{\frac{n-2\sigma}{n+2\sigma}}(u) \le \bar J_{\frac{n-2\sigma}{n+2\sigma}}.
\]
By item (iii) of Lemma \ref{lem:properties},
\[
a'(t)= \frac{2(n+2\sigma)}{n-2\sigma}  M_2(t) \ge 0.
\]
Thus $a(t)\ge a(0)= J_{\frac{n-2\sigma}{n+2\sigma}}(u_0)$. The lemma is proved.
\end{proof}

\begin{lem}\label{lem:global-exist} We have $T^*=\infty$.

\end{lem}

\begin{proof} By \eqref{eq:main}, we have
\[
\pa_t (e^{\int_0^t a(s)\,\ud s} u(t)^m)= e^{-\int_0^t a(s)\,\ud s} \mathcal{K}_{g_0}(u)(t)\ge 0.
\]
If $T^*<\infty$,
\[
e^{\int_0^{T^*} a(s)\,\ud s/m} u(T^*) := \lim_{t\to T^*} e^{m^{-1} \int_0^t a(s)\,\ud s} u(t)
\]
 exists and belongs to $L^{m+1}(M)$. Integrating equation \eqref{eq:main}, we obtain
\be\label{eq:mild-form}
u(t)^m= u_0^m +\int_0^t e^{-\int_s^t a(\tau)\,\ud \tau} \mathcal{K}_{g_0}(u)(s)\,\ud s .
\ee
By Lemma \ref{lem:normalized contant},
\[
u(t)^m\le  u_0^m +\int_0^t e^{J_{\frac{n-2\sigma}{n+2\sigma}}(u_0) \cdot (s-t)} \mathcal{K}_{g_0}(u)(s)\,\ud s.
\]
It follows from Lemma \ref{lem:tail}, Proposition \ref{prop:small_to_regular} and  a bootstrap argument that
\[
\|u^m\|_{C(M\times [0,T^*])} <\infty.
\]
This contradicts to the definition of $T^*$. Hence, $T^*$ can not be a finite positive number. The proof is thus finished.

\end{proof}

Next, we compute the derivative of $M_q$. By Lemma \ref{lem:properties},
\begin{align}
&\frac{\ud }{\ud t }M_q(t) \nonumber \\&=\int_{M} [q |Q_{K_0}^g  -a|^{q-2} (Q_{K_0}^g -a)\pa_t (Q_{K_0}^g  -a) \nonumber \\& \qquad + \frac{2n}{n-2\sigma} |Q_{K_0}^g-a|^{q} (Q_{K_0}^g -a)] \,\ud vol_g  \nonumber \\&
= \int_{M}  \Big[ \frac{(n+2\sigma)q}{n-2\sigma } |Q_{K_0}^g  -a|^{q-2} (Q_{K_0}^g -a)  \mathcal{K}_g (Q_{K_0}^g-a) \nonumber  \\& \qquad +(\frac{2n}{n-2\sigma}-q)|Q_{K_0}^g-a|^{q} (Q_{K_0}^g  -a)  - aq|Q_{K_0}^g-a|^{q}   \label{eq:M-q}  \\&  \qquad  -q a'|Q_{K_0}^g  -a|^{q-2} (Q_{K_0}^g -a)\Big]\,\ud vol_g .  \nonumber
\end{align}
Denote the first term as
\[
N_q:= \int_{M}  \frac{(n+2\sigma)q}{n-2\sigma } |Q_{K_0}^g  -a|^{q-2} (Q_{K_0}^g -a)  \mathcal{K}_g (Q_{K_0}^g-a)\,\ud vol_g,
\]
which is a `bad' term to us because of \eqref{eq:strock-f}.  Using Hardy-Littlewood-Sobolev inequality,
\be \label{eq:riesz-curv}
\begin{split}
|N_q|&=\frac{(n+2\sigma)q}{n-2\sigma }  |\int_{M} |Q_{K_0}^g  -a|^{q-2} (Q_{K_0}^g -a)  \mathcal{K}_g (Q_{K_0}^g-a)   \,\ud vol_g|
\\ &=
 \frac{(n+2\sigma)q}{n-2\sigma }  |\int_{M} |Q_{K_0}^g  -a|^{q-2} u(Q_{K_0}^g -a)  \mathcal{K}_{g_0} (u(Q_{K_0}^g-a))   \,\ud vol_{g_0}|\\&
 \le \frac{(n+2\sigma)q}{n-2\sigma }   C \||Q_{K_0}^g  -a|^{q-2} u(Q_{K_0}^g -a)\|_{L^{\frac{2n}{n+2\sigma}}}\| u(Q_{K_0}^g -a)\|_{L^{\frac{2n}{n+2\sigma}}}\\&
= \frac{(n+2\sigma)q}{n-2\sigma }  C M_{\frac{2n(q-1)}{n+2\sigma}}^{\frac{n+2\sigma}{2n}} M_{\frac{2n}{n+2\sigma}} ^{\frac{n+2\sigma}{2n}}.
\end{split}
\ee
Furthermore, for $q>2$ and $\nu\ge \frac{2n(q-1)}{n+2\sigma}$, by the H\"older inequality  with using $V(t)=1$,
\be \label{eq:riesz-curv-a}
\begin{split}
|N_q| &\le C M_{\frac{2n(q-1)}{n+2\sigma}}^{\frac{n+2\sigma}{2n}} M_{\frac{2n}{n+2\sigma}} ^{\frac{n+2\sigma}{2n}} \\&
\le C M_{\nu}^{\frac{q-1}{\nu}}M_2^{\frac{1}{2}}  \\&
\le  \va M_\nu+\frac{C}{\va^{\frac{q-1}{\nu-q+1}}} M_2^{\frac{\nu}{2(\nu-q+1)}},
\end{split}
\ee
where $\va>0$ can be very small, and the Young inequality is used in the last inequality. As for the second term, since $Q_{K_0}^g  >0$, we have  $(Q_{K_0}^g  -a) \ge -\bar J_{\frac{n-2\sigma}{n+2\sigma}}$ and thus
\be \label{eq:good-term}
|Q_{K_0}^g-a|^{q} (Q_{K_0}^g  -a) \ge |Q_{K_0}^g-a|^{q+1}- \bar J_{\frac{n-2\sigma}{n+2\sigma}} |Q_{K_0}^g-a|^{q}.
\ee
So it is a `good' term to us. As for the last term, we have the estimate, using Lemma \ref{lem:normalized contant},
\be \label{eq:last-term}
 \left|\int_M q a'|Q_{K_0}^g  -a|^{q-2} (Q_{K_0}^g -a) \,\ud vol_g \right| \le  C M_{2} M_{q-1}.
\ee

\begin{prop}\label{prop:cc-0} We have
\[
M_q(t)\to 0 \quad \mbox{as }t\to \infty, \quad \mbox{if } 1\le q < \frac{2n}{n-2\sigma}+\frac{n+2\sigma}{n-2\sigma}.
\]

\end{prop}

\begin{proof}

\textit{Step 1.} We consider  $q=2$.

By \eqref{eq:riesz-curv} and the H\"older inequality, we have
\begin{align*}
|N_2|\le C  M_2.
\end{align*}
Noting that
\[
\int_{M} a' (Q_{K_0}^g -a)  \,\ud vol_g=0,
\]
by \eqref{eq:M-q} and \eqref{eq:good-term} we obtain
\[
\frac{4\sigma }{n-2\sigma}M_3  -C M_2\le \frac{\ud }{\ud t }M_2(t) \le \frac{4\sigma }{n-2\sigma}M_3  +C M_2,
\]
which implies that
\[
\int_1^\infty M_3 \,\ud t \le \frac {n-2\sigma} {4\sigma }\left( M_2(1)+C\int_1^\infty M_2\,\ud t\right )
\]
and
\[
\|\frac{\ud }{\ud t }M_2(t)\|_{L^1( [1,\infty))}\le C \int_{1}^\infty (M_2+M_3)\,\ud t<\infty.
\] Therefore,   $\lim_{t\to \infty} M_2=0$.

\textit{Step 2.} We consider $2< q\le \frac{2n}{n-2\sigma}$, which implies $\frac{2n(q-1)}{n+2\sigma}\le q$.

By taking $\nu=q$ and $\va=1$ in \eqref{eq:riesz-curv-a}, we obtain
\[
|N_q|\le M_q +  C M_2 ^{\frac{q}{2}} \le M_q +  C M_2.
\]
It follows from \eqref{eq:M-q},  \eqref{eq:good-term} and \eqref{eq:last-term} that
\be \label{eq:prop-cc-1}
\begin{split}
- C( M_q +  & M_{2} (1+M_{q-1} ))+(\frac{2n}{n-2\sigma}-q) M_{q+1} \\&  \le  \frac{\ud }{\ud t }M_q(t) \le (\frac{2n}{n-2\sigma}-q) M_{q+1} + C( M_q +   M_{2} (1+M_{q-1} )) .
\end{split}
\ee

If, in addition, $q\le 3 $ , then $M_q\le M_3+M_2 \in L^1([1,\infty))$ and $M_{q-1}\le M_2^{\frac{q-1}{2}}\le C$. Using the left part of the above inequality first, we have
\[
M_{q+1}\in L^1([1,\infty)), \quad q\le 3\mbox{ and }q<\frac{2n}{n-2\sigma}.
\]
Hence, both the lower and upper bound of $\frac{\ud }{\ud t}M_q$ in \eqref{eq:prop-cc-1} belong to $L^1([1,\infty))$, so does it.
Repeating this process, we will conclude that
\begin{align*}
\lim_{t\to \infty}M_{q}(t) = 0& \quad \mbox{for all } 2\le q\le \frac{2n}{n-2\sigma}, \\
M_{q+1}\in L^1([1,\infty))& \quad \mbox{for all } 2\le q< \frac{2n}{n-2\sigma}.
\end{align*}

\textit{Step 3.} We consider  $\frac{2n}{n-2\sigma}<q< \frac{2n}{n-2\sigma} +\frac{n+2\sigma}{n-2\sigma}$, which implies $ \frac{2n(q-1)}{n+2\sigma} \le  q+1$.

By taking $\nu= \frac{2n(q-1)}{n+2\sigma} $  and $\va=\frac12( q-\frac{2n}{n-2})$ in \eqref{eq:riesz-curv-a}, in view of that
\[
\frac{\nu}{2(\nu-q+1)}= \frac{n(q-1)}{2n(q-1)- (q-1)(n+2\sigma)}= \frac{n}{n-2\sigma}>1,
\]
then we have
\begin{align*}
|N_q|\le \va M_\nu+ C M_2^{\frac{n}{n-2\sigma}} & \le \va (M_{q+1} +M_2) + CM_2^{\frac{n}{n-2\sigma}}  \\& \le \va M_{q+1} +C M_2\\&
= \frac12( q-\frac{2n}{n-2}) M_{q+1} + CM_2.
\end{align*}
It follows from \eqref{eq:M-q},  \eqref{eq:good-term} and \eqref{eq:last-term}  that
\be \label{eq:prop-cc-2}
\begin{split}
- C( M_q +  & M_{2} (1+M_{q-1} )) \\&  \le  \frac{\ud }{\ud t }M_q(t) +\frac12(q-\frac{2n}{n-2\sigma}) M_{q+1} \le  C( M_q +   M_{2} (1+M_{q-1} )) .
\end{split}
\ee
Arguing as in Step 2,  we will again conclude that $M_q(t)\to 0$ as $t\to \infty$.  The proposition is proved.
\end{proof}

Let $D J_{\frac{n-2\sigma}{n+2\sigma}}(f, \cdot): L^{\frac{2n}{n-2\sigma}}(M)\to \R$ be the Frech\'et differential of the functional $J_{\frac{n-2\sigma}{n+2\sigma}}$ at $f\in L^{\frac{2n}{n+2\sigma}}(M)$.

\begin{cor}\label{cor:PS} Along the flow,
\[
D J_{\frac{n-2\sigma}{n+2\sigma}}(u, \cdot ) \to 0\quad \mbox{as }t\to \infty.
\]
Hence, the flow is a Palais-Smale flow line.
\end{cor}

\begin{proof} For any $\varphi\in L^{\frac{2n}{n+2\sigma}}(M)$, we have
\begin{align*}
D J_{\frac{n-2\sigma}{n+2\sigma}}(u, \varphi)&= 2 \int_{M} (\mathcal{K}_{g_0}(u)-a(t) u^{\frac{n-2\sigma}{n+2\sigma}}) \varphi )\,\ud vol_{g_0}\\&
= 2  \int_{M} (Q_{K_0}^g -a) u^{\frac{n-2\sigma}{n+2\sigma}} \varphi \,\ud vol_{g_0}
\le 2 M_{\frac{2n}{n-2\sigma}}^{\frac{n-2\sigma}{2n}} \|\varphi \|_{L^{\frac{2n}{n+2\sigma}}(M)}.
\end{align*}
The corollary follows from Proposition \ref{prop:cc-0}.
\end{proof}

\begin{proof}[Proof of Theorem \ref{thm:cc}]
 It follows from Lemma \ref{lem:normalized contant} and Proposition \ref{prop:cc-0}.
\end{proof}

\subsection{Global bound via the moving spheres method}

\begin{thm} \label{thm:harnack} Suppose that $(M,g_0)$ is the standard sphere, $\mathcal{K}_{g_0}= (P_{\sigma}^{g_0})^{-1}$ in \eqref{eq:p-inverse}. If $u$ is positive solution of \eqref{eq:main} on $\Sn\times (0,\infty)$ and $u(0)^m\in C^1(\Sn)$ is not identical to zero,  then $u\in C^1(\Sn \times (0,\infty))$ and the differential Harnack inequality holds
\[
|\nabla \ln u| \le  C \quad \mbox{on }\Sn\times [1, \infty),
\]
and thus
\[
\frac{1}{C}\le u\le C  \quad \mbox{on }\Sn\times [1, \infty),
\]
where $C>0$ depends on $u(t)$ with $t\in [1/2, 1]$.
\end{thm}

The global existence follows from \ref{lem:global-exist} and the $C^1$ regularity follows from a bootstrap argument for the equation \eqref{eq:mild-form}.

Next, we shall use the moving spheres method in  Li  \cite{Li04}  to prove the differential Harnack inequality.   Pick any point $\xi_0\in \mathbb{S}^n$ as the south pole and let $F$ be the inverse of the stereographic projection with  Jacobi determinant $|J_F|$. Let
\[
v(x)= |J_F|^{\frac{n-2\sigma}{2n}} u(F(x))^{\frac{n-2\sigma}{n+2\sigma}} .
\]
Then we have
\be \label{eq:main-Rn}
\frac{1}{A(t)} \pa_t [A(t) v(x,t)]= c_{n,\sigma} \int_{\R^n} \frac{v(y,t)^{\frac{n+2\sigma}{n-2\sigma}}}{|x-y|^{n-2\sigma}}\,\ud y  \quad \mbox{in }\R^n\times [0,\infty),
\ee
where
\[
A(t)= e^ { \int_0^t a(s)\,\ud s}.
\]
For $\lda>0$ and $x_0\in\R^n $, denote
\[
v_{x_0,\lda}(x,t)= \left(\frac{\lda}{|x-x_0|}\right)^{n-2\sigma} v (x^{x_0,\lda},t), \quad \mbox{where } x^{x_0,\lda}= x_0+\frac{\lda^2 (x-x_0)}{|x-x_0|^2}
\]
as the generalized Kelvin transform of $v$ with respect to the sphere $\pa B_\lda(x_0)$.
Using the following two identities (see, e.g., page 162 of \cite{Li04}),
\be \label{eq:ID-1}
\left(\frac{\lda}{|x-x_0|}\right)^{n-2\sigma} \int_{|z-x_0|\ge \lda} \frac{v(z,t)^{\frac{n+2\sigma}{n-2\sigma}}}{|x^{x_0,\lda}-z|^{n-2\sigma}}\,\ud z= \int_{|z-x_0|\le \lda} \frac{v_{x_0,\lda}(z,t)^{\frac{n+2\sigma}{n-2\sigma}}}{|x-z|^{n-2\sigma}}\,\ud z
\ee
and
\be \label{eq:ID-2}
\left(\frac{\lda}{|x-x_0|}\right)^{n-2\sigma} \int_{|z-x_0|\le \lda} \frac{v(z,t)^{\frac{n+2\sigma}{n-2\sigma}}}{|x^{x_0,\lda}-z|^{n-2\sigma}}\,\ud z= \int_{|z-x_0|\ge \lda} \frac{v_{x_0,\lda}(z,t)^{\frac{n+2\sigma}{n-2\sigma}}}{|x-z|^{n-2\sigma}}\,\ud z,
\ee
it is easy to see that $v_{x_0,\lda}$ is also a solution of \eqref{eq:main-Rn} in $(\R^n\setminus \{x_0\} )\times [0,\infty)$.
Notice that
\be \label{eq:reflect-1}
\begin{split}
&\frac{1}{A(t)}\pa_t  [A(t)(v(x,t)- v_{x_0, \lda}(x,t))]\\&= \int_{|z-x_0|\ge \lda} K (x_0, \lda; x,z) [v(z,t)^{\frac{n+2\sigma }{n-2\sigma}} -v_{x_0,\lda}(z,t)^{\frac{n+2\sigma }{n-2\sigma}} ] \,\ud z, \quad x\in \R^n \setminus \overline B_\lda(x_0),
\end{split}
\ee
where
\[
K (x_0, \lda; x,z) = \frac{1}{|x-z|^{n-2\sigma}}- (\frac{\lda}{|x-x_0|})^{n-2\sigma} \frac{1}{|x^{x_0,\lda}- z|^{n-2\sigma}}.
\]
It is elementary to check that
\begin{align*}
K (x_0, \lda; x,z)& >0, \quad  \forall~ |x-x_0|, |z-x_0|>\lda>0 ,\\
K (x_0, \lda; x,z)& =0, \quad  \forall~ |x-x_0|=\lda,\\
\nabla_x K (x_0, \lda; x,z) \cdot (x-x_0) & >0,  \quad  \forall~ |x-x_0|=\lda,~ |z-x_0|>\lda.
\end{align*}

\begin{lem} \label{lem:prep-1} There exist positive  constants $\lda_0$ and $\va_0$ such that for each $x_0\in B_1$, there holds
\be \label{eq:start-1}
v_{x_0,\lda}(x,t) < v(x,t),  \quad \forall~  0<\lda\le \lda_0, \quad |x-x_0| > \lda, \quad t\in [\frac12,1],
\ee
\be \label{eq:start-2}
 v(x,t)- v_{x_0,\lda}(x,t) \ge \frac{\va_0}{ |x|^{n-2\sigma}} \quad \forall~|x|\ge \lda_0+1, \quad t\in [\frac12,  1],
\ee
and
\be \label{eq:start-3}
 v(x,t)- v_{x_0,\lda}(x,t) \ge \va_0 (|x-x_0| -\lda ) \quad \forall~  \lda \le |x|\le \lda_0+1, \quad t\in [\frac34,  1].
\ee
\end{lem}

\begin{proof} This follows from a direct computation. See the proof of Lemma 3.1 of  Li  \cite{Li04}.

\end{proof}

We shall show that \eqref{eq:start-1} holds for all $t\in [1,\infty)$.

Fix an arbitrary $T>\frac32$ and a point $x_0\in B_1$.  Without loss of generality, we assume $x_0=0$ and write $v_\lda= v_{0,\lda}$ for brevity. Similar to Lemma \ref{lem:prep-1}, we have the following lemma which asserts that one can start the moving spheres procedure up to $T$.

\begin{lem}\label{lem:prep-2}  There exists $\lda_T\in (0,\lda_0]$ depending on $T$ such that
\[
v_{\lda}(x,t) < v(x,t),  \quad \forall~  0<\lda\le \lda_T, \quad |x| \ge \lda, \quad \frac34\le t\le T.
\]
\end{lem}

Define
\[
\bar \lda= \sup\{\mu\le \lda_0: v_{\lda}(x,t) \le v(x,t),  \quad \forall~  0<\lda\le \mu, \quad |x| \ge \lda, \quad 1\le t\le T \}.
\]
Obviously, $\bar \lda\ge \lda_T$.

\begin{lem}\label{lem:prep-3} There exists $\va_2>0 $ such that
\[
 v(x,t)- v_{\bar \lda}(x,t) \ge \frac{\va_2}{ |x|^{n-2\sigma}} \quad \forall~|x|\ge \bar \lda+1, \quad t\in [1,  T],
\]
and
\[
 v(x,t)- v_{\bar \lda}(x,t) \ge \va_2 (|x| -\bar \lda ) \quad \forall~ \bar  \lda \le |x|\le \bar \lda+1, \quad t\in [1,  T].
\]
\end{lem}

\begin{proof} Let
\[
\xi(z,t,\lda)=\begin{cases}
\frac{v(z,t)^{\frac{n+2\sigma }{n-2\sigma}} -v_{\lda}(z,t)^{\frac{n+2\sigma }{n-2\sigma}}} {v(z,t)-v_{\lda}(z,t) },& \quad v(z,t)\neq v_{\lda}(z,t),\\
0,& \quad v(z,t)=v_{\lda}(z,t),
\end{cases}
\]
and $w^\lda(z,t)= v(z,t)-v_{\lda}(z,t) $. By \eqref{eq:reflect-1}, we have
\[
w^{\bar \lda}(x,t)=A(t)^{-1} w^{\bar \lda}(x,\frac34) + \int_{\frac34}^t\int_{B_\lda(x_0)^c}  \frac{A(s)}{A(t)}  K (x_0, \bar \lda; x,z)w(z,t)^{\bar \lda}\,\ud z\ud s
\]
for $(x,t)\in B_\lda(x_0)^c\times (\frac{3}{4}, T] $. Since $w^\lda \ge 0$, we obtain
\[
w^{\bar \lda}(x,t)\ge A(t)^{-1} w^{\bar \lda}(x,\frac34)  + \int_{\frac34}^1\int_{B_\lda(x_0)^c}  \frac{A(s)}{A(t)}  K (x_0, \bar \lda; x,z)w(z,t)^{\bar \lda}\,\ud z\ud s.
\]
The lemma then follows from Lemma \ref{lem:prep-1}.

\end{proof}

\begin{lem}\label{lem:prep-4}  We have  $\bar \lda=\lda_0$.

\end{lem}

\begin{proof} If not, by the above lemma,
\[
\pa_r [(v(\cdot, t)- v_{\bar \lda}(x,t) ]\Big|_{|x|=\bar \lda} \ge \va_2 \quad \mbox{for }t\in [\frac34,T].
\]
By the continuity of $\nabla v$, there exists a small $\va_3>0$ so that
\[
\pa_r [(v(\cdot, t)- v_{ \lda}(x,t) ]\Big|_{ \pa B_r } \ge \va_2/2 \quad \mbox{for } \bar \lda \le \lda,r \le \bar \lda+\va,  ~\frac34 \le t \le T.
\]
Since $v(\cdot, t)- v_{ \lda}(x,t)=0$ on $\pa B_{\lda}$, we have
\[
v(\cdot, t)- v_{ \lda}(x,t) >0 \quad \mbox{for }  \bar \lda \le \lda < |x| \le  \bar \lda +\va_3,  ~\frac34 \le t \le T.
\]
Using the first lower bound in Lemma \ref{lem:prep-3} and choosing $\lda-\bar \lda $ to be very  small,
\[
v(\cdot, t)- v_{ \lda}(x,t) >0 \quad \mbox{for }  |x| >\bar \lda +\va_3 ,  ~\frac34 \le t \le T.
\] We obtain a contradiction and the lemma follows.

\end{proof}

\begin{proof}[Proof of Theorem \ref{thm:harnack}]  By the above lemmas, we have, for all $x_0\in B_1$,
\[
v(x,t) \ge v_{x_0,\lda}(x,t) \quad \forall ~ 0<\lda<\lda_0,  |x-x_0|\ge \lda, ~ t\in [\frac{3}{4}, T].
\]
By Lemma A.1 and Lemma A.2 of \cite{LL}, we  have
\[
|\nabla \ln v(t)| \le C \quad \mbox{in }B_{1/2}, \quad \frac{3}{4}\le t\le T.
\]
Since $\xi_0\in \Sn$ and $T$ are arbitrarily chosen, the differential Harnack inequality follows. Since the flow keeps the volume, the uniform positive lower and upper bounds follow. The proof is finished.
\end{proof}

\section{Convergence}

\label{s:conver}

It is important to mention that the normalized flow, which preserves the volume $\int_{M} u^{m+1}\,\ud ovl_{g_0}$, is equivalent to \eqref{eq:main-rescaled} upon a variable change. In this section, we aim to demonstrate the convergence of solutions of \eqref{eq:main-rescaled}, including the scenario where $m>1$ with replacing the $\tau$ variable by  $\tau= \ln(t+1) $ in \eqref{eq:renorm}, provided that the solutions are consistently bounded between positive constants. Namely, suppose that
\begin{equation} \label{eq:main-rescaled-g}
\pa_t u ^m=  \mathcal{K}_{g_0}( u ) -\frac{m}{|1-m|} u^m \quad \mbox{on }M\times (0,\infty)
\end{equation}
and
\be \label{eq:bd-assump}
\frac{1}{C_0}\le u\le C_0 \quad \mbox{on }M\times (0,\infty),
\ee
where $K_0$ satisfies \eqref{eq:K-1}-\eqref{eq:K-3}, $C_0\ge 1$ is a constant and $m\in (0,1)\cup (1,\infty)$. We shall prove that $u$ converges to a steady solution
\be \label{eq:steady-1}
\mathcal{K}_{g_0} (\varphi)=\frac{m}{|1-m|} \varphi^m \quad \mbox{on }M, \quad \varphi>0.
\ee

First we need two lemmas.

\begin{lem} \label{lem:sim-1} Let $\varphi$ be a solution of \eqref{eq:steady-1} and $\zeta=u-\varphi$. Then there exist a constant depends only $M,g, n,\sigma, \Lda$ and $C_0$ such that
\[
\|\zeta(\cdot,t+\tau)\|_{L^2(M)}\le C e^{C t} \|\zeta(\cdot,t)\|_{L^2(M)}, \quad \forall~ t\ge 1,\tau\ge 0.
\]
\end{lem}

\begin{proof} By the equation of $u$ and $\varphi$, a direct computation yields
\be \label{eq:error-1}
m u^{m-1} \pa_t \zeta= \mathcal{K}_{g_0} \zeta + \frac{m}{|1-m|} \eta \zeta,
\ee
where
\[
\eta= \int_0^1m((1-\lda)\varphi+\lda u)^{m-1}\,\ud \lda.
\]
By \eqref{eq:main-rescaled-g} and \eqref{eq:bd-assump}, $\eta$ and  $\pa_t u$ are uniformly bounded.  Multiplying both sides of  \eqref{eq:error-1} by $\zeta$ and integrating over $M$, we have
\[
\frac{\ud }{\ud t} \int_{M} \zeta^2 u^{m-1}  \,\ud vol_{g_0}\le C \int_{M} \zeta^2 u^{m-1}  \,\ud vol_{g_0}.
\]
By applying Gronwall's inequality and taking into account \eqref{eq:bd-assump}, we finish the proof.
\end{proof}

\begin{lem} \label{lem:sim-2} Assume as in Lemma \ref{lem:sim-1}. Then we have
\[
\|\pa_t \zeta(\cdot, t+1)\|_{C^0(M)} + \|\zeta(\cdot, t+1)\|_{C^0(M)} \le C \| \zeta(\cdot, t)\|_{L^2(M)}, \quad t\ge 1,
\]
where $C>0$ depends only $M,g, n,\sigma, \Lda$ and $C_0$.
\end{lem}

\begin{proof} Since $\zeta$ satisfies the linear equation \eqref{eq:error-1}, the lemma follows from a bootstrap argument with using Lemma \ref{lem:sim-1} as a starting point.
\end{proof}

\begin{thm}\label{thm:conditional} Assume as above. Suppose that $u \in C^1([0,\infty); C^0(M))$ is a positive solution of \eqref{eq:main-rescaled-g} satisfying \eqref{eq:bd-assump}. Then
\[
\lim_{t\to \infty} u(t)=S \quad \mbox{uniformly on }M,
\]
where $S$ is a positive solution of \eqref{eq:steady-1}.
\end{thm}

\begin{proof} Since the flow possesses a gradient structure, the idea of Simon \cite{Sim} can be adapted. We follow the proof of Theorem 1.2 in Jin-Xiong-Yang \cite{JXY} with some slight deviation in establishing local estimates.

Let $\beta=\frac{m}{|1-m|}$. By \eqref{eq:main-rescaled-g}, we  have, for any $0\le t_0<t$,
\be \label{eq:equi-0}
u(X, t)^m= e^{-\beta (t-t_0)}u(X, t_0)^m+ \int_{t_0}^{t} e^{\beta(s-t)} \mathcal{K}_{g_0} (u)(X,s)\,\ud s.
\ee
By the potential estimates, we have
\be \label{eq:ab-estimates}
\begin{split}
&u(X, t)^m-u(Y,t)^m\\ & \le e^{-\beta (t-t_0)} (u(X, t_0)^m-u(Y,t_0)^m)+ C |X-Y|^{\al}\int_{t_0}^t e^{\beta(s-t)} \,\ud s\\&
\le e^{-\beta (t-t_0)}(u(X, t_0)^m-u(Y,t_0)^m) + C |X-Y|^{\al}, \quad \forall~ X,Y\in M,
\end{split}
\ee
where $C$ and $\al$ are positive constants depending only on $M, g_0, K_0,m$ and $C_0$ in \eqref{eq:bd-assump}. By taking $t_0=0$, it follows that
$u(t)$ is uniformly continuous for $t\in [0,\infty)$ and there exist a sequence  $t_j\to \infty$ and a positive function $u_\infty\in C(M)$ such that
\[
u(t_j) \to u_\infty \quad \mbox{in }C^0(M), \quad \mbox{as }j\to \infty.
\]

\textit{Claim:}  $u_\infty$ is a solution of \eqref{eq:steady-1}.

Indeed, let
\[
G_m(u)= \int_{M} \Big(\frac12 u\mathcal{K}_{g_0} u-\frac{ \beta}{m+1} u^{m+1}\Big)\,\ud vol_{g_0}.
\]
Then we have
\be \label{eq:g-t}
\frac{\ud }{\ud t} G_m(u)= \int_{M} (\mathcal{K}_{g_0} u-\beta u^m)\pa_t u \,\ud vol_{g_0}= m \int_{M} |\pa_t u|^2 u^{m-1} \,\ud vol_{g_0} \ge 0.
\ee
By the assumption \eqref{eq:bd-assump}, $G(u)$ is bounded and hence,
\[
\lim_{t\to \infty} G(u(t))=G(u_\infty).
\]
For $\tau>0$, we have
\begin{align*}
&\int_{M} |u(t_j+\tau)^{\frac{m+1}{2}}- u(t_j)^{\frac{m+1}{2}}|^2 \,\ud vol_{g_0}\\&= \int_{M}|\int_{t_j}^{t_j+\tau} \pa_s u(t_j+s)^{\frac{m+1}{2}}\,\ud s|^2\,\ud vol_{g_0} \\&
\le \tau \frac{(m+1)^2}{4} \int_{M} \int_{t_j}^{t_j+\tau} |\pa_s u(t_j+s)| u(t_j+s)^{m-1}\,\ud s\ud vol_{g_0}\\&
= \tau \frac{(m+1)^2}{4 m} (G(u(t_j+\tau))-G(u(t_j))).
\end{align*}
Using the pointwise estimate
\[
|u(X, t_j+\tau)^{\frac{m+1}{2}}- u(X, t_j)^{\frac{m+1}{2}}|\le |u(X, t_j+\tau)- u(X, t_j)|^{\frac{m+1}{2}},
\]
we have $u(X, t_j+\tau) \to u_\infty$ in $L^{m+1}$ uniformly in $\tau$. By interpolation inequality, we have, for any $m+1< q<\infty$,
\[
u(X, t_j+\tau) \to u_\infty\quad \mbox{in }L^{q}(M)
\]
uniformly in $\tau$. By \eqref{eq:equi-0},
\[
u(X, t_j+1)^m= e^{-\beta }u(X, t_j)^m+ \int_{t_j}^{t_j+1} e^{\beta(s-t_{j}-1)} \mathcal{K}_{g_0} (u)(X,s)\,\ud s.
\]
Sending $j\to \infty$, we obtain
\[
u_\infty^m =e^{-\beta }u_\infty^m+ \mathcal{K}_{g_0} (u_\infty) \lim_{j\to \infty} \int _{t_j}^{t_j+1} e^{\beta(s-t_{j}-1)}\,\ud s,
\]
i.e.,
\[
\beta u_\infty^m= \mathcal{K}_{g_0} (u_\infty).
\]
The claim is verified.

Since  the functional $G(\cdot)$ is real analytic on
\[
\om_C= \{f\in L^{m+1}(M): \frac{1}{C}\le  f\le C\},
\]
the so-called Lojasiewicz-Simon gradient inequality holds; see Chill \cite{Chill}. Then one can use the standard argument to  show that
\[
u(t)\to u_\infty \quad \mbox{in }C(M), \quad \mbox{as }t\to \infty.
\]
In fact, armed with Lemmas  \ref{lem:sim-1} and \ref{lem:sim-2}, one can mimic the corresponding proof in \cite{JXY} to establish the aforementioned full convergence. We omit the details. The theorem is proved.

\end{proof}

As in \cite{JXY},  the linearized operator at $S$ will play a crucial role in the convergence rate. We may consider the eigenvalue problem
\be \label{eq:steady-2}
\mathcal{K}_{g_0} (\phi)=\lda S^{m-1} \phi \quad \mbox{on }M.
\ee
To seek a symmetry structure, we introduce
\[
\ud \mu= S^{1-m} \ud vol_{g_0}, \quad \mathcal{K}^{\mu}(f)(X)= \int_{M} K_0(X,Y)f(Y)\,\ud \mu
\]
and $L^2(M,\ud \mu)$ space equipped with the inner product
\[
\langle f, h\rangle_{L^2(M,\ud \mu)}= \int_{M} fh \,\ud \mu.
\] Note that
\[
\langle \mathcal{K}^{\mu}(f), h\rangle_{L^2(M,\ud \mu)}= \langle f, \mathcal{K}^{\mu}(h)\rangle_{L^2(M,\ud \mu)}
\] and $\mathcal{K}^{\mu}: L^2(M,\ud \mu) \to L^2(M,\ud \mu)$ is compact.  The eigenvalue problem
\be \label{eq:steady-3}
\mathcal{K}^{\mu}(\varphi) =\lda \varphi \quad \mbox{in }L^2(M, \ud \mu),
\ee
has countable many eigenvalues, which must be real.
If $\varphi$ is an eigenfunction, then $\phi= S^{1-m} \varphi$ will be an eigenfunction of \eqref{eq:steady-2}. In line with \cite{JXY}, one can establish a sharp convergence rate. We leave the details to the interested reader.

\begin{proof}[Proof of Theorem \ref{thm:main0}] If $m>1$, it follows from Proposition \ref{prop:mge1}. If $\frac{n-2\sigma}{m+2\sigma}<m<1$, it follows from Proposition \ref{prop:fde-2} and Theorem \ref{thm:conditional}. If $m=\frac{n-2\sigma}{m+2\sigma}$, by  Proposition \ref{prop:fde-1} we have the lower and upper bound. Let $\tilde u$ be defined as in \eqref{eq:main-rescaled} and 
\[
G(\tilde u)= \int_{M} \left[\frac12\tilde u\mathcal{K}_{g_0} \tilde u -\frac{m}{(1-m)(1+m)} \tilde u^{m+1}\right].
\] By direct computation, we have 
\[
\frac{\ud }{\ud t}G=\int_{M} \pa_t \tilde u^m \pa_t \tilde u\ge 0.
\] Since $G$ is bounded, $\lim_{t\to \infty} G=G_\infty $ exists. 
On the other hand, 
\[
\quad \frac{m}{m+1} \frac{\ud }{\ud t }\int_{M} \tilde u^{m+1}= 2 G +\frac{m}{1+m} \int_{M} \tilde u^{m+1} . 
\]
Thus 
\[
\int_{M} \tilde u(t)^{m+1}= e^{t} \big(\int_{M} \tilde u_0^{m+1} +\frac{2(m+1)}{m} \int_0^t e^{-s}G(\tilde u(s))\,\ud s\big).
\]
Since $\int_{M} \tilde u(t)^{m+1}$ is bounded, this forces   
\[
\frac{2(m+1)}{m} \int_0^\infty e^{-s}G(\tilde u(s))\,\ud s=-\int_{M} \tilde u_0^{m+1} .
\] It follows that
\[
\int_{M} \tilde u(t)^{m+1}= -\frac{2(m+1)}{m} \int_t^{\infty} e^{t-s} G(\tilde u(s))\,\ud s\to -\frac{2(m+1)}{m}  G_\infty,
\]
as $t\to \infty$. We complete the proof  of Theorem \ref{thm:main0}.
\end{proof}

\begin{proof}[Proof of Theorem \ref{thm:main1}] By  Theorem \ref{thm:harnack}, we have global positive upper and lower bounds. Using Lemma \ref{lem:normalized contant} and a change of variable, we can transform the normalized flow into the \eqref{eq:main-rescaled-g}.  Theorem \ref{thm:main1} then follows from Theorem \ref{thm:conditional}.

\end{proof}

\end{document}